\documentclass[british]{article}
\usepackage{amsmath}
\usepackage{amsthm}
\usepackage{amssymb}
\usepackage{fontspec}
\usepackage{xcolor}
\usepackage{mathtools}
\usepackage{stmaryrd}
\usepackage[all]{xy}
\usepackage[pdfusetitle,
 bookmarks=true,bookmarksnumbered=false,bookmarksopen=false,
 breaklinks=true,pdfborder={0 0 0},pdfborderstyle={},backref=false,colorlinks=true]
 {hyperref}
\hypersetup{
 colorlinks=true,linkcolor=magenta,urlcolor=blue,citecolor=magenta}

\makeatletter
\theoremstyle{plain}
\newtheorem{thm}{\protect\theoremname}[section]
\theoremstyle{plain}
\newtheorem{question}[thm]{\protect\questionname}
\theoremstyle{definition}
\newtheorem{defn}[thm]{\protect\definitionname}
\theoremstyle{plain}
\newtheorem{lem}[thm]{\protect\lemmaname}
\theoremstyle{remark}
\newtheorem{rem}[thm]{\protect\remarkname}
\theoremstyle{plain}
\newtheorem{cor}[thm]{\protect\corollaryname}
\theoremstyle{remark}
\newtheorem{notation}[thm]{\protect\notationname}
\theoremstyle{plain}
\newtheorem{prop}[thm]{\protect\propositionname}
\theoremstyle{definition}
\newtheorem{example}[thm]{\protect\examplename}

\usepackage{fullpage}
\date{}
\AtBeginDocument{}

\usepackage{tikz-cd}
\usetikzlibrary{arrows}
\tikzset{commutative diagrams/.cd,arrow style=tikz,diagrams={>=latex'}}

\makeatother

\usepackage{polyglossia}
\setdefaultlanguage[variant=british]{english}
\usepackage[style=alphabetic]{biblatex}
\providecommand{\corollaryname}{Corollary}
\providecommand{\definitionname}{Definition}
\providecommand{\examplename}{Example}
\providecommand{\lemmaname}{Lemma}
\providecommand{\notationname}{Notation}
\providecommand{\propositionname}{Proposition}
\providecommand{\questionname}{Question}
\providecommand{\remarkname}{Remark}
\providecommand{\theoremname}{Theorem}

\begin{document}
\global\long\def\rg{R\llbracket G\rrbracket}%
\global\long\def\fp#1{\mathrm{FP}_{#1}}%
\global\long\def\Ker{\ker}%
\global\long\def\operatorname#1{\mathrm{#1}}%
\global\long\def\F{\mathbb{F}}%
\global\long\def\Z{\mathbb{Z}}%
\global\long\def\scarytimes{\boxtimes}%
\global\long\def\di{\dim}%
\global\long\def\dimf{\dim_{\mathbb{F}}}%
\global\long\def\d{\mathrm{d}}%
\global\long\def\llb{\llbracket}%
\global\long\def\rrb{\rrbracket}%
\global\long\def\Hom{\mathrm{Hom}}%
\global\long\def\rep#1{\mathcal{S}^{#1}}%
\global\long\def\supG#1{\sup_{\text{simple }#1\llbracket G\rrb\text{-modules }M}}%
\global\long\def\supRG{\supG R}%
\global\long\def\cb#1#2{\tau_{#1}(#2)}%
\global\long\def\cbf#1#2#3{\tau_{#1}(#2;#3)}%
\global\long\def\ct{\otimes}%
\global\long\def\t{\ct_{R}}%
\global\long\def\te{\ct_{\rg}}%
\global\long\def\li{\underset{\longleftarrow}{\lim}}%
\global\long\def\torg#1{\mathrm{Tor}^{\rg}_{#1}}%

\title{Virtual Surjection and the $n$-$(n+1)$-$(n+2)$ Theorem for Profinite
Groups}
\author{Tal Cohen and Mark Shusterman}
\maketitle
\begin{abstract}
We prove the Virtual Surjection Conjecture for profinite groups. Namely,
given a product of $n$ profinite $\mathrm{FP}_{k}$ groups, a subgroup
that virtually surjects onto $k$-tuples must be $\mathrm{FP}_{k}$
as well. We also prove the $n$-$(n+1)$-$(n+2)$ Conjecture for profinite
groups, as well as a few other $\mathrm{FP}_{n}$ permanence results
for fibre products. Our main tool is a numerical criterion for property
$\mathrm{FP}_{n}$ of modules of profinite groups. Our work suggests
a new finiteness property to investigate. 
\end{abstract}

\section{Introduction}

Finiteness properties of groups are a central theme in geometric and
homological group theory. For a discrete group $\Gamma$, one says
that $\Gamma$ is of type $\mathrm{F}_{n}$ if it admits a classifying
space (an Eilenberg–MacLane space $K(\Gamma,1)$) with finite $n$-skeleton.
Property $\mathrm{F}_{1}$ is equivalent to finite generation, and
$\mathrm{F}_{2}$ to finite presentability. Homologically, one says
that $\Gamma$ is of type $\mathrm{FP}_{n}$ (over $\mathbb{Z}$)
if the trivial $\mathbb{Z}\Gamma$-module $\mathbb{Z}$ admits a projective
resolution that is finitely generated in degrees $0$ through $n$.
The geometric condition $\mathrm{F}_{n}$ implies the homological
condition $\mathrm{FP}_{n}$; property $\mathrm{F}_{1}$ is equivalent
to property $\mathrm{FP}_{1}$, but it is a classical result of Bestvina
and Brady \cite{BestvinaBrady97} that this fails for $n\geqslant2$.
For finitely presented groups, however, property $\mathrm{FP}_{n}$
is equivalent to property $\mathrm{F}_{n}$. Therefore, establishing
property $\mathrm{F}_{n}$ often boils down to property $\mathrm{FP}_{n}$.

In the realm of profinite groups, the homological analogue $\mathrm{FP}_{n}$
takes centre stage. A profinite group $G$ is said to be of type $\mathrm{FP}_{n}$
over a commutative profinite ring $R$ if the trivial $R\llbracket G\rrbracket$-module
$R$ admits a projective resolution in the category of profinite modules
that is finitely generated in degrees $0$ through $n$ (see Section
\ref{sec:NumericalCriterion} for precise definitions). More generally,
one may speak of a profinite $R\llbracket G\rrbracket$-module $K$
being of type $\mathrm{FP}_{n}$, and property $\mathrm{FP}_{n}$
for the group $G$ is the special case $K=R$.

Property $\mathrm{FP}_{n}$ for profinite groups has attracted substantial
attention. The foundations were laid in the pro-$p$ setting by King
\cite{King99}, and the theory was developed for general profinite
groups by Corob Cook \cite{Coo16}, who established Bieri–Eckmann-type
criteria and proved closure under extensions, quotients by $\mathrm{FP}_{n}$
subgroups, proper amalgamated free products, and proper HNN extensions.
Probabilistic aspects were explored by Corob Cook and Vannacci \cite{CoVa}.
Property $\mathrm{FP}_{n}$ for profinite groups has been extensively
studied in the works of Weigel, Zalesskii, Kochloukova, Pinto and
others. For instance, property $\mathrm{FP}_{n}$ for profinite metabelian
and solvable groups has been investigated.

The relevance of profinite group theory to discrete group theory is
exhibited, among other things, in profinite rigidity: for certain
classes of residually finite groups, the profinite completion may
determine the group up to isomorphism. Property $\mathrm{FP}_{n}$
of discrete and profinite groups has played a role in the study of
profinite rigidity, for instance in \cites{KocZal}{Alex14}{Jai20}{GarJai}{JaiIsmael}.

The closer analogues of discrete groups in the profinite world are
pro-$p$ groups. Many fundamental results in discrete group theory
have pro-$p$ counterparts, but their proofs in the pro-$p$ setting
often require different arguments. Indeed, the flow of ideas has sometimes
gone from the pro-$p$ world to the discrete one: for instance, Serre
\cite{Serre65} proved in 1965 that a virtually free torsion-free
pro-$p$ group is actually free, and the analogous result for discrete
groups was established only later, by Stallings \cite{Stallings68}
and Swan \cite{Swan69}. As another example, free and surface groups
satisfy the virtual retraction property in both the discrete and the
pro-$p$ settings, as opposed to the profinite setting.

For general profinite groups, the landscape is considerably richer
than for pro-$p$ groups. For example, whereas subgroups of discrete
free groups and free pro-$p$ groups are always free, this is no longer
the case for free profinite groups. Moreover, Howson's theorem that
the intersection of finitely generated subgroups of free groups is
finitely generated, which is true both in the discrete and pro-$p$
settings, fails for free profinite groups \cites{JaikinMO}{ZaZa}.
As for property $\mathrm{FP}_{n}$, a pro-$p$ group is $\mathrm{FP}_{n}$
if and only if $H^{i}(G,\mathbb{F}_{p})$ is finite for every $i\leqslant n$
\cite[Theorem 1.6]{PLETCH198055}, but it was shown by Corob Cook
\cite[§5]{Coo16} that the same does not hold for profinite groups. 

The present paper sets the stage for systematically extending results
on property $\mathrm{FP}_{n}$ from pro-$p$ groups to general profinite
groups. We adopt the paradigm of proving theorems about groups by
establishing more general facts about modules. An example of this
approach is \cite{Alj}, providing an extension of the aforementioned
theorem by Serre to modules. Our principal tool is the following family
of numerical invariants.

Let $G$ be a profinite group, $R$ a commutative profinite ring,
and $K$ a profinite $R\llbracket G\rrbracket$-module. For each nonnegative
integer $n$, we define 
\[
\cbf nGK=\sup_{\text{simple }R\llbracket G\rrbracket\text{-modules }M}\frac{\dim_{\mathrm{End}_{G}(M)}\mathrm{Tor}^{R\llbracket G\rrbracket}_{n}(K,M)}{\dim_{\mathrm{End}_{G}(M)}M}.
\]
This formula takes a simpler form when working over a field with
the trivial module:
\[
\cbf nG{\mathbb{F}}=\supG{\mathbb{F}}\frac{\dim_{\mathbb{F}}H_{n}(G,M)}{\dim_{\mathbb{F}}M}.
\]
For discrete groups, one may similarly define $\cbf nG{\mathbb{F}}$
(taking the supremum over simple finite-dimensional $\mathbb{F}[G]$-modules).
If $G$ is a discrete $\mathrm{FP}_{n}$ group, then $\cbf iG{\mathbb{F}}<\infty$
for every $i\leqslant n$, but the converse does not hold. In the
profinite world, however, property $\mathrm{FP}_{n}$ is captured
completely by these numerical invariants.
\begin{thm}[Numerical criterion; Theorem \ref{thm:numericalFP} and Theorem \ref{thm:NumericalCriterion2}]
Let $G$ be a profinite group, $R$ a commutative profinite ring,
$K$ a profinite $R\llbracket G\rrbracket$-module, and $n$ a nonnegative
integer or $\infty$. Then $K$ is of type $\mathrm{FP}_{n}$ if and
only if $\cbf iGK<\infty$ for every $i\leqslant n$.
\end{thm}

This theorem may be viewed as an extended version of \cite[Corollary 5.7]{CoVa},
applicable to general profinite modules. A similar criterion also
appeared in \cite[Proposition 2.2]{Sym}. 

Many of the results in this paper are, at their core, quantitative
statements about the invariants $\tau_{n}$, from which the corresponding
results on property $\mathrm{FP}_{n}$ follow via this numerical criterion.
In fact, similar proofs will establish the same bounds on these invariants
for discrete groups as well. It seems likely that our methods can
be extended to the setting of pseudocompact algebras \cite{PseudoBRUMER}.
While this does not imply anything regarding property $\mathrm{FP}_{n}$
for discrete groups, it suggests investigating the finiteness of $\cb nG$
as a property of its own. \\

Alexander Lubotzky's paper \cite{Lubotzky01} studied profinite presentations
by analysing the interplay between the number of generators, the number
of relations, and the dimensions of cohomology groups with coefficients
in simple modules. Our numerical criterion can be viewed as a higher-dimensional
analogue of these foundational results. \cite{Lubotzky01} relates
the minimal number of generators and relators of a profinite presentation
to invariants that are similar to our $\tau_{1}(G)$ and $\tau_{2}(G)$.
Our criterion extends this approach to all degrees, capturing property
$\mathrm{FP}_{n}$ through the finiteness of $\tau_{0},\tau_{1},\ldots,\tau_{n}$. 

The case $n=1$ is already visible, in somewhat different language,
in Erika Damian's work on the augmentation ideal, where she gives
a formula for its minimal number of generators in terms of simple
modules \cite{Damian11}. Our framework unifies and extends this to
all degrees.

We record several further properties of the invariants $\tau_{n}$,
for context:
\begin{itemize}
\item \emph{Finite presentability:} If $G$ is a finitely generated profinite
group, then it is finitely presented if and only if $\tau_{2}(G)$
is finite \cites[Theorem 0.3]{Lubotzky01}, if and only if $G$ is
$\mathrm{FP}_{2}$ (Theorem \ref{thm:numericalFP}).
\item \emph{$\mathrm{FP}_{1}$ and finite generation:} If $G$ is a prosolvable
(profinite) group, then it is finitely generated if and only if it
is $\mathrm{FP}_{1}$ \cite[Proposition 3.4 and Remark 3.5]{Coo16},
if and only if $\tau_{1}(G)<\infty$ (Theorem \ref{thm:numericalFP}).
\item \emph{Nilpotent groups:} finitely generated virtually nilpotent profinite
groups are $\mathrm{FP}_{\infty}$ (Theorem \ref{thm:FGNilpotentFP}). 
\item \emph{Augmentation ideal}: A profinite group is $\mathrm{FP}_{1}$
if and only if its augmentation ideal is finitely generated (Corollary
\ref{cor:augmentation}).
\end{itemize}
For commutative rings, $\mathrm{FP}_{n}$ modules have been studied
under the name of $n$-pseudocoherent. The topological setting has
been considered in particular in \cites{kedlaya-liu-imperfect-period-rings}{scholze-condensed-mathematics}.
It is conceivable that a numerical criterion may play a role in that
context as well.

\subsection{Main results}

With the numerical criterion in hand, we develop a systematic theory
of $\mathrm{FP}_{n}$ for profinite groups, establishing quantitative
versions of a number of results. Our two main results resolve profinite
analogues of two open problems about discrete groups from geometric
group theory.

The first concerns subdirect products. Given groups $G_{1},\dots,G_{n}$
of type $\mathrm{FP}_{k}$ and a subgroup $P\leqslant G_{1}\times\cdots\times G_{n}$
that virtually surjects onto every product of $k$ factors, one asks
whether $P$ is necessarily of type $\mathrm{FP}_{k}$. The pro-$p$
case was proved in \cite{propnn1n2}. We prove it for profinite groups:
\begin{thm}[Virtual Surjection Theorem; Theorem \ref{thm:VirtualSurjection}]
Let $2\leqslant k\leqslant n$, let $G_{1},\ldots,G_{n}$ be profinite
groups of type $\mathrm{FP}_{k}$, and let $P\leqslant G_{1}\times\cdots\times G_{n}$
be a closed subgroup whose projection to any product of $k$ factors
is open. Then $P$ is of type $\mathrm{FP}_{k}$.
\end{thm}

The key ingredient in the proof of the Virtual Surjection Theorem
is a result on fibre products, which is a profinite analogue of the
$n$-$(n+1)$-$(n+2)$ Conjecture. Here too the pro-$p$ case was
solved in \cite{propnn1n2}. 
\begin{thm}[The $n$-$(n+1)$-$(n+2)$ Theorem; Theorem \ref{thm:n-n+1-n+2}]
\label{thm:intronnn}Let $n$ be a nonnegative integer or $\infty$,
let 
\[
1\to N_{1}\to G_{1}\to Q\to1,\qquad1\to N_{2}\to G_{2}\to Q\to1
\]
be short exact sequences of profinite groups, and let $K$ be an $R\llbracket Q\rrbracket$-module.
Then, for every nonnegative integer $k\leqslant n$,
\[
\cbf k{G_{1}\times_{Q}G_{2}}K\leqslant\cbf k{G_{1}}K+\sum^{k+1}_{r=2}\cbf rQK\cbf{k+1-r}{N_{1}}R+\sum_{p+q=k,p>0}\cbf p{G_{2}}K\cbf q{N_{1}}R.
\]
In particular, if $N_{1}$ is $\mathrm{FP}_{n-1}$, $G_{1}$ and $G_{2}$
are $\mathrm{FP}_{n}$, and $Q$ is $\mathrm{FP}_{n+1}$, then the
fibre product $G_{1}\times_{Q}G_{2}$ is $\mathrm{FP}_{n}$. 
\end{thm}

These questions were originally studied for discrete groups and property
$\mathrm{F}_{k}$ as well as $\mathrm{FP}_{k}$. Bridson, Howie, Miller
and Short \cite{BHMS} established the “Virtual Surjection to Pairs”
theorem (with property $\mathrm{F}_{2}$). Kuckuck \cite{Kuc14} formulated
the $n$-$(n+1)$-$(n+2)$ Conjecture and showed that it implies the
Virtual Surjection Conjecture with property $\mathrm{F}_{k}$ (and
a variant of $\mathrm{FP}_{k}$). Then Kochloukova–Lima \cite{KocLim}
studied these questions for discrete groups and property $\mathrm{FP}_{k}$,
establishing the Virtual Surjection Conjecture in the case $k=2$
(for property $\mathrm{FP}_{k}$). The general case, for both property
$\mathrm{FP}_{k}$ and $\mathrm{F}_{k}$, remains an open question
for discrete groups. We remark that the $n$-$(n+1)$-$(n+2)$ Conjecture
for property $\mathrm{FP}_{k}$ implies the $n$-$(n+1)$-$(n+2)$
Conjecture for property $\mathrm{F}_{k}$, since the $0$-$1$-$2$
case is known, and, for $\mathrm{F}_{2}$ groups, property $\mathrm{FP}_{k}$
is equivalent to property $\mathrm{F}_{k}$.

We follow Kuckuck's strategy \cite{Kuc14}, and deduce the Virtual
Surjection Theorem from the $n$-$(n+1)$-$(n+2)$ Theorem. 

We remark that Theorem \ref{thm:intronnn} implies an $n$-$(n+1)$-$(n+2)$
result for modules of profinite groups: if $N_{1}$ is $\mathrm{FP}_{n-1}$,
$K$ is $\mathrm{FP}_{n}$ as an $R\llbracket G_{1}\rrbracket$-module,
$K$ is $\mathrm{FP}_{n}$ as an $R\llbracket G_{2}\rrbracket$-module,
and $K$ is $\mathrm{FP}_{n+1}$ as an $R\llbracket Q\rrbracket$-module,
then $K$ is $\mathrm{FP}_{n}$ as an $R\llbracket G_{1}\times_{Q}G_{2}\rrbracket$-module.

The $n$-$\left(n+1\right)$-$\left(n+2\right)$ conjecture has also
been considered in other settings, for example in \cite{Koc19} for
Lie algebras.\\

We observe that, in both formulations of the numerical criterion (Theorem
\ref{thm:numericalFP} and \ref{thm:NumericalCriterion2}), only the
value of $\lceil\cbf nGK\rceil$ matters. This raises the question
of what further information, if any, is carried by the precise value
of $\cbf nGK$.
\begin{question}
Can $\tau_{n}(G)$ be irrational? More generally, for which profinite
groups $G$ and modules $K$ is $\cbf nGK$ guaranteed to be rational? 
\end{question}

\subsection{Organisation}

The paper is organised as follows. In Section \ref{sec:NumericalCriterion},
we define the invariants $\cbf nGK$ and establish the numerical criterion
for property $\mathrm{FP}_{n}$, Theorem \ref{thm:numericalFP} (see
also Theorem \ref{thm:NumericalCriterion2}). In Section \ref{sec:BasicOps},
we use the numerical criterion to prove standard closure properties:
behaviour under extensions, open subgroups, quotients, amalgamated
free products, and HNN extensions. In particular, we recover (with
explicit bounds) the results of Corob Cook \cite{Coo16} and extend
them to general modules. We also weaken the requirement on the kernel
in the passage of property $\mathrm{FP}_{n}$ to quotients; see Lemma
\ref{lem:FPforQuotients}. We extend Serre's definition of goodness
of discrete groups to modules, and show that the profinite completion
of a good $\mathrm{FP}_{n}$ module is $\mathrm{FP}_{n}$; see Section
\ref{subsec:Good} and Proposition \ref{prop:Good} for more details.
In Section \ref{sec:FibreProducts}, we prove a few other fibre-product
permanence results concerning property $\mathrm{FP}_{n}$. Section
\ref{sec:Nilpotent} treats nilpotent and polyprocyclic groups. Section
\ref{sec:DirectProducts} studies irreducible representations of direct
products. Section \ref{sec:nn1n2} contains the proof of the $n$-$(n+1)$-$(n+2)$
Theorem. Lastly, in Section \ref{sec:VirtualSurjection} we prove
the Virtual Surjection Theorem

\bigskip

\paragraph{Conventions.}

Unless otherwise stated, all modules we consider are left profinite
modules. Given a profinite ring $\Lambda$, we denote by $\mathrm{PMOD}(\Lambda)$
the abelian category of left profinite $\Lambda$-modules. When $V$
is a vector space over some field $\mathbb{F}$, we denote by $V^{*}$
the linear dual of $V$ over $\mathbb{F}$. We denote by $\otimes$
the completed tensor product.

\section{Numerical Criterion for $\mathrm{FP}_{n}$ \label{sec:NumericalCriterion}}
\begin{defn}
Let $R$ be a commutative profinite ring, $G$ a profinite group,
$n$ a nonnegative integer. We say that an $\rg$-module $K$ is of
type $\mathrm{FP}_{n}$ if there exists an exact sequence 
\[
P_{n}\to\cdots\to P_{1}\to P_{0}\to K\to0,
\]
where $P_{0},\dots,P_{n}$ are finitely generated projective $\rg$-modules.
We say that $K$ is of type $\mathrm{FL}_{n}$ if there exists such
a sequence with $P_{0},\dots,P_{n}$ free. We say that $K$ is of
type $\mathrm{FP}_{\infty}$ (respectively, $\mathrm{FL}_{\infty}$)
if $K$ is of type $\mathrm{FP}_{n}$ (respectively, $\mathrm{FL}_{n}$)
for every $n\geqslant0$.

We say that $G$ is $\mathrm{FP}_{n}$ (respectively $\mathrm{FL}_{n}$)
over $R$ if the trivial module $R$ is $\mathrm{FP}_{n}$ (respectively
$\mathrm{FL}_{n}$). Observe that, if $G$ is $\mathrm{FP}_{n}$ (respectively
$\mathrm{FL}_{n}$) over $\hat{\mathbb{Z}}$, then it is $\mathrm{FP}_{n}$
(respectively $\mathrm{FL}_{n}$) over every commutative profinite
ring $R$, as one can see by tensoring the resolution with $R$.
\end{defn}

Recall that, when $M$ is a simple $G$-module, $\mathrm{End}_{G}(M)$
is a finite field and $M$ itself is a finite-dimensional vector space
over $\mathrm{End}_{G}(M)$. 

\begin{defn}
For a profinite group $G$, a commutative profinite ring $R$, an
$\rg$-module $K$ and a nonnegative $n$, we set
\[
\cbf nGK=\supRG\frac{\dim_{\mathrm{End}_{G}(M)}\mathrm{Tor}^{\rg}_{n}(K,M)}{\dim_{\mathrm{End}_{G}(M)}M}.
\]
When $R$ is clear from the context, we simply write $\cb nG$ for
$\cbf nGR$ (where $R$ is the trivial $G$-module). For negative
$n$, we set $\cbf n{\cdot}{\cdot}=0$,
\end{defn}

Observe that 
\[
\frac{\log\left|\mathrm{Tor}^{\rg}_{n}(K,M)\right|}{\log\left|M\right|}\leqslant\cbf nGK
\]
for every finite $\rg$-module $M$ (by induction on the length).

For intuition, we record the following special case of Lemma \ref{lem:TorSpec}.
\begin{lem}
\label{lem:flat-modules}When $K$ is flat as an $R$-module, we have
\[
\cbf nGK=\supRG\frac{\dim_{\mathrm{End}_{G}(M)}H_{n}(G,K\otimes_{R}M)}{\dim_{\mathrm{End}_{G}(M)}M}
\]
\end{lem}

\begin{rem}
If, moreover, $K=R=\mathbb{F}$ is a finite field (which we consider
as a one-dimensional trivial $G$-module), we have:
\[
\cbf nG{\mathbb{F}}=\supG{\mathbb{F}}\frac{\dim_{\mathbb{F}}H_{n}(G,M)}{\dim_{\mathbb{F}}M}=\supG{\mathbb{F}}\frac{\dim_{\mathbb{F}}H^{n}(G,M)}{\dim_{\mathbb{F}}M}.
\]
\end{rem}

If $A$ is an $R\llbracket G\rrbracket$-module, we denote by $\d_{R\llbracket G\rrbracket}A$
the minimal number of generators of $A$ (as an $R\llbracket G\rrbracket$-module).
Given an $\rg$-module $N$, we denote by $\mathrm{Rad}(N)$ the \emph{radical
}of $N$, i.e., the intersection of all maximal open submodules of
$N$. The quotient $N/\mathrm{Rad}(N)$ is a (possibly infinite) direct
product of simple $\rg$-modules. Given a simple $\rg$-module $M$,
we denote by $i_{N}(M)$ the ‘number of times’ $M$ appears
in the direct product $N/\mathrm{Rad}(N)$ (i.e., the largest cardinality
$\kappa$ such that $N/\mathrm{Rad}(N)$ surjects onto $M^{\kappa}=\prod_{\kappa}M$).

\begin{thm}[{\cite[Theorem 5.6]{CoVa}}]
\label{prop:no-of-gen}Let $G$ be a profinite group, $R$ a commutative
profinite ring, $N$ an $\rg$-module. Then 
\[
\d_{\rg}(N)=\supG R\left\lceil \frac{i_{N}(M)}{\dim_{\mathrm{End}_{G}(M)}M}\right\rceil .
\]
\end{thm}

\begin{lem}
\label{lem:min-gen-rank-by-tau}Let $G$ be a profinite group, $R$
a commutative profinite ring, $K$ an $\rg$-module. Then 
\[
\mathrm{d}_{\rg}(K)=\left\lceil \cbf 0GK\right\rceil .
\]
\end{lem}

\begin{proof}
First suppose that $K$ can be generated by $d$ elements, so that
we have a surjection $\rg^{d}\longrightarrow K$. Since tensoring
is right exact, this supplies us, for every simple $\rg$-module $M$,
with a surjection 
\[
\rg^{d}\te M=M^{d}\longrightarrow K\te M,
\]
so that 
\[
\dim_{\mathrm{End}_{G}(M)}(K\te M)\leqslant d\dim_{\mathrm{End}_{G}(M)}M.
\]
Since $M$ was arbitrary, it follows that $\left\lceil \cbf 0GK\right\rceil \leqslant\mathrm{d}_{\rg}K$. 

For the other direction, let $M$ again be a simple $\rg$-module
and set $\mathbb{F}=\mathrm{End}_{G}(M)$. Denote $\kappa=i_{K}(M)$,
so that $K/\mathrm{Rad}(K)$ surjects onto $M^{\kappa}$. Since tensoring
is right exact, we get surjections 
\[
M^{*}\otimes_{\llbracket RG\rrbracket}K\longrightarrow M^{*}\te\left(K/\mathrm{Rad}(K)\right)\longrightarrow M^{*}\te M^{\kappa},
\]
where we denote by $(\cdot)^{*}$ the linear dual over $\mathbb{F}$.
Observe that
\begin{align*}
\left(M^{*}\otimes_{R\llbracket G\rrbracket}M\right)^{*} & =\left(\left(M^{*}\otimes_{R}M\right)_{G}\right)^{*}=\left(\left(M^{*}\otimes_{R}M\right)^{*}\right)^{G}=\mathrm{Hom}_{\mathbb{F}}(M^{*}\otimes_{R}M,\mathbb{F})^{G}\\
 & =\mathrm{Hom}_{R}(M,\mathrm{Hom}_{\mathbb{F}}(M^{*},\mathbb{F}))^{G}=\mathrm{Hom}_{R}(M,M^{**})^{G}\\
 & =\mathrm{Hom}_{R}(M,M)^{G}=\mathrm{End}_{G}(M),
\end{align*}
so
\[
\kappa\leqslant\dim_{\mathbb{F}}(M^{*}\te K).
\]
Therefore
\[
\d_{\rg}(K)=\supG R\left\lceil \frac{\kappa}{\dim_{\mathbb{F}}M}\right\rceil \leqslant\supG R\left\lceil \frac{\dim_{\mathbb{F}}(M^{*}\te K)}{\dim_{\mathbb{F}}M}\right\rceil \leqslant\left\lceil \cbf 0GK\right\rceil .\qedhere
\]
\end{proof}

\subsection{Free and Projective Resolutions}
\begin{thm}
\label{thm:numericalFP}Let $G$ be a profinite group, $R$ a commutative
profinite ring, $K$ an $\rg$-module, $n$ a nonnegative integer
or $\infty$. Then $K$ is $\mathrm{FL}_{n}$ if and only if $K$
is $\mathrm{FP}_{n}$, if and only if $\cbf kGK<\infty$ for every
nonnegative integer $k\leqslant n$. Numerically:
\begin{enumerate}
\item \label{enu:Numeric}If $\cbf kGK$ is finite for every $k\leqslant n$,
then we have an exact sequence
\[
\cdots\to L_{1}\to L_{0}\to K\to0,
\]
where $L_{k}$ is a free $\rg$-module of rank at most $\sum^{k}_{i=0}\left\lceil \cbf iGK\right\rceil $
for every $k\leqslant n$.
\item \label{enu:Standard}If we have an exact sequence 
\[
\cdots\to P_{1}\to P_{0}\to K\to0,
\]
where $P_{k}$ is a projective $\rg$-module for every $k\leqslant n$,
then $\cbf kGK\leqslant\d_{R\llbracket G\rrbracket}P_{k}$ for every
$k\leqslant n$.
\end{enumerate}
\end{thm}

The main ingredient of the proof is the following lemma. 
\begin{lem}
\label{lem:ExtendingResolutions}Let $G$ be a profinite group, $R$
a commutative profinite ring, $K$ an $\rg$-module and $i$ a nonnegative
integer. Suppose $\cbf{i+1}GK$ is finite and that
\[
P_{i}\to\cdots\to P_{0}\to K\to0
\]
is an exact sequence, where $P_{0},\dots,P_{i}$ are projective $\rg$-modules.
Set $N=\ker(P_{i}\to P_{i-1})$ (or $N=\ker(P_{0}\to K)$ if $i=0$).
Then
\[
\cbf 0GN\leqslant\cbf{i+1}GK+\mathrm{d}_{R\llbracket G\rrbracket}P_{i}.
\]
\end{lem}

\begin{proof}
We prove this by induction on $i$. Consider first the case $i=0$.
We have an exact sequence
\[
0\to N\to P_{0}\to K\to0.
\]
Let $M$ be a simple $\rg$-module, which is a finite-dimensional
vector space over the finite field $\mathbb{F}\coloneqq\mathrm{End}_{G}(M)$.
The long exact sequence of $\mathrm{Tor}$ gives us an exact sequence
\begin{align*}
\mathrm{Tor}^{\rg}_{1}(K,M)\to N\te M\to P_{0}\te M,
\end{align*}
which means 
\begin{align*}
\dim_{\mathbb{F}}N\te M & \leqslant\dim_{\mathbb{F}}\mathrm{Tor}^{\rg}_{1}(K,M)+\dim_{\mathbb{F}}P_{0}\te M,
\end{align*}
so that
\[
\frac{\dim_{\mathbb{F}}N\te M}{\dim_{\mathbb{F}}M}\leqslant\cbf 1GK+\mathrm{d}_{R\llbracket G\rrbracket}P_{0}.
\]
Since $M$ was arbitrary, the claim follows. Now, assume the claim
is true for $i-1$, and let us prove it for $i$. Set $N_{0}=\ker(P_{0}\to K)$,
so that
\[
P_{i}\to\cdots\to P_{1}\to N_{0}\to0
\]
is exact. By the induction hypothesis, we have 
\[
\cbf 0GN\leqslant\cbf iG{N_{0}}+\mathrm{d}_{R\llbracket G\rrbracket}P_{i}.
\]
On the other hand, the exact sequence 
\[
0\to N_{0}\to P_{0}\to K\to0
\]
again supplies us with, for every simple $R\llbracket G\rrbracket$-module
$M$, the exact sequence 
\begin{align*}
0 & \to\mathrm{Tor}^{\rg}_{i+1}(K,M)\to\mathrm{Tor}^{\rg}_{i}(N_{0},M)\to0\to\cdots\\
\cdots & \to\mathrm{Tor}^{\rg}_{2}(K,M)\to\mathrm{Tor}^{\rg}_{1}(N_{0},M)\to0\to\mathrm{Tor}^{\rg}_{1}(K,M)\to N_{0}\te M\to P_{0}\te M\to K\te M\to0.
\end{align*}
Since $i>0$, this implies $\mathrm{Tor}^{R\llbracket G\rrbracket}_{i+1}(K,M)\cong\mathrm{Tor}^{R\llbracket G\rrbracket}_{i}(N_{0},M)$,
which means $\cbf iG{N_{0}}=\cbf{i+1}GK$, so we are done.
\end{proof}

\begin{proof}[Proof of Theorem \ref{thm:numericalFP}]
First, assume $\cbf iGK<\infty$ for every $i\leqslant n$. We construct
the resolution by recursion on $i=0,\dots,n-1$. The case $i=0$ is
Lemma \ref{lem:min-gen-rank-by-tau}; assume by induction that we
have an exact sequence
\[
L_{i}\to\cdots\to L_{0}\to K\to0,
\]
where $i<n$ and where $L_{i}$ is a free $\rg$-module of rank at
most $\sum^{i}_{j=0}\left\lceil \cbf jGK\right\rceil $. By Lemma
\ref{lem:ExtendingResolutions}, $N\coloneqq\ker(L_{i}\to L_{i-1})$
can be generated by $\sum^{i}_{j=0}\left\lceil \cbf jGK\right\rceil +\left\lceil \cbf{i+1}GK\right\rceil $
elements, so we may construct a map $\rg{}^{\sum^{i+1}_{j=0}\left\lceil \cbf jGK\right\rceil }\longrightarrow L_{i}$
so that 
\[
\rg^{\sum^{i+1}_{j=0}\left\lceil \cbf jGK\right\rceil }\to L_{i}\to\cdots\to L_{0}\to K\to0
\]
is an exact sequence of free $\rg$-modules, as needed.

Now, suppose we have an exact sequence
\[
\cdots\to P_{1}\to P_{0}\to K\to0
\]
where each $P_{i}$ is a projective $\rg$-module. Let $M$ be a simple
$\rg$-module. For $k\leqslant n$, $\torg k(K,M)$ is the $k^{\text{th}}$
homology group of
\[
\cdots\to P_{1}\te M\to P_{0}\te M\to0,
\]
so that 
\[
\dim_{\mathrm{End}_{G}(M)}\torg k(K,M)\leqslant\dim_{\mathrm{End}_{G}(M)}P_{k}\te M.
\]
We have a surjection $\rg^{\d_{\rg}P_{k}}\longrightarrow P_{k}$,
and hence a surjection 
\[
\rg^{\d_{\rg}P_{k}}\te M=M^{\d_{\rg}P_{k}}\longrightarrow P_{k}\te M
\]
(since tensoring with $M$ is right exact). Putting these together,
we get
\[
\dim_{\mathrm{End}_{G}(M)}\torg k(K,M)\leqslant\dim_{\mathrm{End}_{G}(M)}P_{k}\te M\leqslant\mathrm{d}_{\rg}P_{k}\cdot\dim_{\mathrm{End}_{G}(M)}M,
\]
and we are done.
\end{proof}

We record a corollary which is well-known:
\begin{cor}
\label{cor:augmentation}If $G$ is a profinite group, then $G$ is
$\mathrm{FP}_{1}$ if and only if its augmentation ideal is finitely
generated.
\end{cor}

\begin{proof}
If the augmentation ideal $I_{R}(G)=\ker(R\llbracket G\rrbracket\to R)$
can be generated by $d$ elements, we can construct an exact sequence
\[
\rg^{d}\to R\llbracket G\rrbracket\to R\to0,
\]
so that $G$ is $\mathrm{FP}_{1}$. Conversely, if $G$ is $\mathrm{FP}_{1}$
then $\cbf 1GR$ is finite, so the fact 
\[
R\llbracket G\rrbracket\to R\to0
\]
is exact and Lemma \ref{lem:ExtendingResolutions} imply $\ker(R\llbracket G\rrbracket\to R)$
is finitely generated.
\end{proof}

In \cite{Lubotzky01}, it was proved that a finitely generated profinite
group is finitely presented if and only if $\cb 2G$ is finite. We
therefore get:
\begin{cor}
Let $G$ be a finitely generated profinite group. Then $G$ is finitely
presented if and only if it is $\mathrm{FP}_{2}$, if and only if
$\tau_{2}(G)$ is finite.
\end{cor}

\subsection{Resolutions by Projective Covers}

In this subsection we consider a particular projective resolution
of an $R\llbracket G\rrbracket$-module $K$, for which $P_{k}$ is
finitely generated if and only if $\cbf kGK<\infty$.

Let $R$ be a commutative profinite ring, $G$ a profinite group,
$K$ an $R\llbracket G\rrbracket$-module. A \emph{projective cover
}is a surjective map $f:P\to K$ from a projective profinite $R\llbracket G\rrbracket$-module
$P$ such that the following holds: if $H\subseteq P$ is a submodule
for which $f(H)=K$, then $H=P$. Equivalently, one says that $\ker f$
is \emph{superfluous}, in the following sense: if $H\subseteq P$
is a submodule such that $H+\ker f=P$, then $H=P$. Projective covers
exist and are unique up to isomorphism (see \cite[Remark 3.4.3(i)]{ProjCov}).
Therefore, we may construct a canonical resolution of $K$, by first
taking a projective cover $\bar{P}_{0}\longrightarrow K$, and then
taking a projective cover $\bar{P}_{n+1}\longrightarrow K_{n}$ of
the kernel $K_{n}\coloneqq\ker(\bar{P}_{n}\longrightarrow\bar{P}_{n-1})$
for every $n\geqslant0$ (setting $\bar{P}_{-1}=K$).
\begin{lem}
\label{lem:tensor-of-projective-covers}Let $G$ be a profinite group,
$R$ a commutative profinite ring, $K$ an $R\llbracket G\rrbracket$-module,
$P$ a projective profinite $R\llbracket G\rrbracket$-module, and
$p:P\to K$ a projective cover. If $M$ is a finite simple $R\llbracket G\rrbracket$-module,
then 
\[
\mathrm{Hom}_{G}(P,M)\cong\mathrm{Hom}_{G}(K,M)
\]
and 
\[
P\te M\cong K\te M.
\]
\end{lem}

\begin{proof}
We have a natural map $\mathrm{Hom}_{G}(K,M)\to\mathrm{Hom}_{G}(P,M)$,
which is injective since $p$ is surjective. Since $p$ is a projective
cover, this map is moreover surjective: given any map $f:P\to M$,
we have $f(\ker p)=0$ (otherwise we would have $f(\ker p)=M$, by
simplicity of $M$, which implies $\ker p+\ker f=P$, which implies
$\ker f=P$ since $\ker p$ is superfluous, which is a contradiction).
This means $f$ induces a map $\bar{f}:K\to M$, which means $f$
is in the image of $\mathrm{Hom}_{G}(K,M)\to\mathrm{Hom}_{G}(P,M)$.

Let $\mathbb{F}=\mathrm{End}_{G}(M)$, which is a finite field, and
denote by $(-)^{\vee}=\mathrm{Hom}_{\mathbb{F}}(-,\mathbb{F})$ the
continuous linear dual over $\mathbb{F}$. We have $(P\te M)^{\vee}\cong\mathrm{Hom}_{G}(P,M^{\vee})$
and $(K\te M)^{\vee}\cong\mathrm{Hom}_{G}(K,M^{\vee})$. By the previous
paragraph, we have $\mathrm{Hom}_{G}(P,M^{\vee})\cong\mathrm{Hom}_{G}(K,M^{\vee})$
(since $M^{\vee}=M^{*}$ is also simple). Therefore, $(P\te M)^{\vee}\cong(K\te M)^{\vee}$,
so 
\[
P\te M=(P\te M)^{\vee\vee}\cong(K\te M)^{\vee\vee}=K\te M.\qedhere
\]
\end{proof}

\begin{lem}
\label{lem:tau-of-projective-cover}Let $G$ be a profinite group,
$R$ a commutative profinite ring, $K$ an $R\llbracket G\rrbracket$-module,
$k$ a nonnegative integer. Let 
\[
P_{k}\to\cdots\to P_{0}\to K\to0
\]
be a resolution of $K$ by projective covers, and let $M$ be a simple
$R\llbracket G\rrbracket$-module. Set $N=\ker(P_{k}\to P_{k-1})$
(or $N=\ker(P_{0}\to K)$ in the case $k=0$). Then 
\[
\mathrm{Tor}^{\rg}_{r}(N,M)\cong\mathrm{Tor}^{\rg}_{r+k+1}(K,M)
\]
for every nonnegative integer $r$.
\end{lem}

\begin{proof}
We prove this by induction on $k$, starting with the base case $k=0$.
We have an exact sequence
\[
0\to N\to P_{0}\to K\to0,
\]
which supplies us with an exact sequence
\begin{align*}
0 & \to\mathrm{Tor}^{\rg}_{r+1}(K,M)\to\mathrm{Tor}^{\rg}_{r}(N,M)\to0\to\cdots\\
\cdots & \to\mathrm{Tor}^{\rg}_{2}(K,M)\to\mathrm{Tor}^{\rg}_{1}(N,M)\to0\to\mathrm{Tor}^{\rg}_{1}(K,M)\to N\te M\to P_{0}\te M\to K\te M\to0
\end{align*}
(since $\mathrm{Tor}^{\rg}_{1}(P_{0},M)=0$, since $P_{0}$ is projective).
This means $\mathrm{Tor}^{\rg}_{r+1}(K,M)\cong\mathrm{Tor}^{\rg}_{r}(N,M)$
for $r>0$. Moreover, by Lemma \ref{lem:tensor-of-projective-covers},
$P_{0}\te M\to K\te M$ is an isomorphism, which means $\mathrm{Tor}^{\rg}_{1}(K,M)\cong N\te M$.

Now, for the inductive step, set $N_{0}=\ker(P_{0}\to K)$, so that
we have an exact sequence
\[
P_{k}\to\cdots\to P_{1}\to N_{0}\to0.
\]
By the previous case, $\mathrm{Tor}^{\rg}_{r+1}(K,M)\cong\mathrm{Tor}^{\rg}_{r}(N_{0},M)$
for every $r\geqslant0$. The claim now follows by the induction hypothesis.
\end{proof}

\begin{lem}
\label{lem:ProjectiveCoverResolution}Let $G$ be a profinite group,
$R$ a commutative profinite ring, $K$ an $\rg$-module and $k$
a nonnegative integer. Suppose 
\[
P_{k}\to\cdots\to P_{0}\to K\to0
\]
is a resolution by projective covers. Set $N=\ker(P_{k}\to P_{k-1})$
(or $N=\ker(P_{0}\to K)$ in the case $k=0$). Then $\d_{R\llbracket G\rrbracket}N=\left\lceil \cbf{k+1}GK\right\rceil $.
\end{lem}

\begin{proof}
By Lemma \ref{lem:tau-of-projective-cover},
\[
\mathrm{Tor}^{\rg}_{0}(N,M)\cong\mathrm{Tor}^{\rg}_{k+1}(K,M),
\]
so the claim follows from Lemma \ref{lem:min-gen-rank-by-tau}.
\end{proof}

\begin{thm}
\label{thm:NumericalCriterion2}Let $G$ be a profinite group, $R$
a commutative profinite ring, $K$ an $\rg$-module and $k$ a nonnegative
integer. Let
\[
\cdots\to P_{1}\to P_{0}\to K\to0
\]
be a resolution by projective covers. Then $\d_{\rg}P_{k}=\left\lceil \cbf kGK\right\rceil $.
In particular, $\cbf kGK$ is finite if and only if there exists a
projective resolution 
\[
P_{k}\to\cdots\to P_{0}\to K\to0
\]
with $P_{k}$ finitely generated.
\end{thm}

\begin{proof}
Set $N=\ker(P_{k-1}\to P_{k-2})$ (or $N=K$ if $k=0$). Since $P_{k}\longrightarrow N$
is a projective cover, we have $\mathrm{d}_{\rg}P_{k}=\mathrm{d}_{\rg}N$.
We also have $\mathrm{d}_{\rg}N=\left\lceil \cbf kGK\right\rceil $
(by Lemma \ref{lem:ProjectiveCoverResolution} for $k>0$ and Lemma
\ref{lem:min-gen-rank-by-tau} for $k=0$), so we are done. 
\end{proof}

\begin{rem}
It follows in particular from the numerical criterion that the projective
dimension of a module $K$ is equal to the maximal $n$ such that
$\cbf nGK\neq0$. 
\end{rem}

\section{Basic Operations\label{sec:BasicOps}}

\begin{notation}
All the lemmas and propositions in Sections \ref{subsec:Some-Homological-Algebra}–\ref{subsec:Operations-on-Modules}
are quantified as follows: $G$ is some profinite group, $R$ is some
commutative profinite ring, $K$ is some profinite $\rg$-module,
$n$ is some nonnegative integer or $\infty$, and $k$ a nonnegative
integer such that $k\leqslant n$. For $n=\infty$, terms such as
$n+1,n-1$ are understood as $\infty$ as well; property $\mathrm{FP}_{r}$
is vacuously true for $r<0$.
\end{notation}

We begin with the following well-known result. As usual, we say a
profinite group is \emph{finitely generated }if it is topologically
finitely generated (i.e., contains a dense subgroup that is finitely
generated in the abstract sense).
\begin{lem}[Finite-generation]
\label{lem:Finitely-generated-is-fp1}We have $\cbf 1GR\leqslant\mathrm{d}(G)$.
In particular, if $G$ is finitely generated, then it is $\mathrm{FP}_{1}$.
\end{lem}

\begin{proof}
Let $M$ be a simple $\rg$-module, so that it is finite-dimensional
over $\mathbb{F}\coloneqq\mathrm{End}_{G}(M$). By duality, we need
to show 
\[
\frac{\dim_{\mathbb{F}}H_{1}(G,M)}{\dim_{\mathbb{F}}M}=\frac{\dim_{\mathbb{F}}H^{1}(G,M)}{\dim_{\mathbb{F}}M}\leqslant\d(G).
\]
Let $S\subseteq G$ be a generating set. Recall that $H^{1}(G,M)=Z^{1}(G,M)/B^{1}(G,M)$.
Thus, it is certainly enough to bound $\frac{\dim_{\mathbb{F}}Z^{1}(G,M)}{\dim_{\mathbb{F}}M}$.
But
\[
Z^{1}(G,M)=\left\{ f:G\to M\middle|f(gh)=f(g)+gf(h)\right\} .
\]
If $f_{1},f_{2}\in Z^{1}(G,M)$ agree on $S$, they must be equal;
therefore, $\dimf Z^{1}(G,M)$ is at most $\left|S\right|\cdot\dimf M$,
as needed.
\end{proof}

\subsection{\label{subsec:Some-Homological-Algebra}Some Homological Algebra}
\begin{lem}
\label{lem:TorSpec}Let $M$ be an $R\llbracket G\rrbracket$-module
and $N\trianglelefteqslant G$ a closed normal subgroup. There is
a convergent first quadrant homology spectral sequence
\[
E^{2}_{pq}=\mathrm{Tor}^{R\llbracket G/N\rrbracket}_{p}(R,\mathrm{Tor}^{R\llbracket N\rrbracket}_{q}(K,M))\implies\mathrm{Tor}^{R\llbracket G\rrbracket}_{p+q}(K,M).
\]
\end{lem}

\begin{proof}
We use the Grothendieck spectral sequence (see \cites[Corollary 5.8.4]{Weibel}).
Let $F$ be the functor from $\mathrm{PMOD}(\rg)$ to $\mathrm{PMOD}(R)$
defined by $F=-\te M$, so that $L_{i}F(K)=\torg i(K,M)$ for every
nonnegative integer $i$. $F$ can be written as a composition, $F=G_{2}\circ G_{1}$,
where $G_{1}$ is the functor from $\mathrm{PMOD}(\rg)$ to $\mathrm{PMOD}(R\llbracket G/N\rrbracket)$
given by $G_{1}=-\hat{\otimes}_{R\llbracket N\rrbracket}M$, and
$G_{2}$ is the functor from $\mathrm{PMOD}(R\llbracket G/N\rrbracket)$
to $\mathrm{PMOD}(R)$ given by $G_{2}=-\hat{\otimes}_{R\llbracket G/N\rrbracket}R=(-)_{G/N}$.
We have (for every nonnegative integer $i$) $L_{i}G_{1}(K)=\mathrm{Tor}^{R\llbracket N\rrbracket}_{i}(K,M)$
(when the latter is considered as an $R\llbracket G/N\rrbracket$-module),
and $L_{i}G_{2}=H_{i}(G/N,\cdot)$, so this composition gives us exactly
the spectral sequence that we need: 
\[
H_{p}(G/N,\mathrm{Tor}^{R\llbracket N\rrbracket}_{q}(K,M))=L_{p}(G_{2})L_{q}(G_{1})(K)\implies L_{p+q}(F)(K)=\torg{p+q}(K,M).
\]
Thus, we just need to show $G_{1}$ sends projective modules to $G_{2}$-acyclic
ones. In other words, we need to show that
\[
H_{i}(G/N,P\hat{\otimes}_{R\llbracket N\rrbracket}M)=0
\]
for every projective $\rg$-module $P$. If $P$ is projective, then
there is some $P'$ such that $L\coloneqq P\oplus P'$ is free, and
\[
H_{i}(G/N,L\hat{\otimes}_{R\llbracket N\rrbracket}M)=H_{i}(G/N,P\hat{\otimes}_{R\llbracket N\rrbracket}M)\oplus H_{i}(G/N,P'\hat{\otimes}_{R\llbracket N\rrbracket}M).
\]
Therefore, it is enough to show 
\[
H_{i}(G/N,L\hat{\otimes}_{R\llbracket N\rrbracket}M)=0
\]
for every free $\rg$-module $L$. Let $L$ be a free $\rg$-module,
so that $L=R\llbracket G\times T\rrbracket\cong R\llbracket G\rrbracket\t R\llbracket T\rrbracket$
for some profinite space $T$, where the $G$-action is from the
left. Recall that $L\ct_{R\llbracket N\rrbracket}M=(L\t M)_{N}$,
where the $G$-action (and hence the $N$-action) on $L\t M$ is the
diagonal one (i.e., $g.(\ell\otimes m)=g\ell\otimes gm$). Recall
that $L\t M$ endowed with the diagonal action is isomorphic to $L\t M$
endowed with the left $G$-action (\cite[Corollary 5.8.2]{RB10}),
so
\[
(L\t M)_{N}=(R\llbracket G\rrbracket\t R\llbracket T\rrbracket\t M)_{N}\cong R\llbracket G/N\rrbracket\t R\llbracket T\rrbracket\t M
\]
endowed with the left $G$-action (since $\rg_{N}=R\llbracket G/N\rrbracket$).
Now, recall that $R\llbracket G/N\rrbracket\t R\llbracket T\rrbracket\t M$
is nothing but $\mathrm{Ind}^{G/N}_{1}(R\llbracket T\rrbracket\t M)$.
Putting it all together, we get that
\[
H_{i}(G/N,G_{1}(L))=H_{i}(G/N,\mathrm{Ind}^{G/N}_{1}(R\llbracket T\rrbracket\t M)).
\]
Therefore, by Shapiro's lemma, we have (for every positive integer
$i$):
\[
H_{i}(G/N,G_{1}(L))=H_{i}(\left\{ 1\right\} ,R\llbracket T\rrbracket\t M)=0.\qedhere
\]
\end{proof}

\begin{lem}
\label{lem:another-ss}Let $N\trianglelefteqslant G$ be a closed
normal subgroup, and assume $N$ acts trivially on $K$, so that $K$
is also an $R\llbracket G/N\rrbracket$-module. Let $M$ be an $R\llbracket G\rrbracket$-module.
Then there is a convergent first quadrant homology spectral sequence
\[
E^{2}_{p,q}=\operatorname{Tor}^{R\llbracket G/N\rrbracket}_{p}\left(K,\operatorname{Tor}^{R\llbracket N\rrbracket}_{q}(R,M)\right)\Longrightarrow\operatorname{Tor}^{R\llbracket G\rrbracket}_{p+q}(K,M).
\]
\end{lem}

\begin{proof}
Consider the functor 
\begin{align*}
F\colon\operatorname{PMOD}(R\llbracket G\rrbracket) & \longrightarrow\operatorname{PMOD}(R)\\
X & \longmapsto K\otimes_{R\llbracket G\rrbracket}X=(K\otimes_{R}X)_{G}.
\end{align*}
Since $K$ is inflated from $G/N$, this functor factors as 
\[
\operatorname{PMOD}(R\llbracket G\rrbracket)\xrightarrow{\;S\;}\operatorname{PMOD}(R\llbracket G/N\rrbracket)\xrightarrow{\;T\;}\operatorname{PMOD}(R),
\]
where 
\[
S(X)=R\llbracket G/N\rrbracket\otimes_{R\llbracket G\rrbracket}X
\]
and 
\[
T(Y)=K\otimes_{R\llbracket G/N\rrbracket}Y.
\]
Indeed, for every $R\llbracket G\rrbracket$-module $X$, there is
a natural isomorphism 
\[
T(S(X))=K\otimes_{R\llbracket G/N\rrbracket}\left(R\llbracket G/N\rrbracket\otimes_{R\llbracket G\rrbracket}X\right)\cong K\otimes_{R\llbracket G\rrbracket}X=F(X).
\]

We claim that $S$ sends projective $R\llbracket G\rrbracket$-modules
to $T$-acyclic $R\llbracket G/N\rrbracket$-modules. Since projective
$R\llbracket G\rrbracket$-module are direct summands of free ones,
it is enough to check this for free modules. If $L$ is a free profinite
$R\llbracket G\rrbracket$-module, then $L\cong R\llbracket G\rrbracket\otimes_{R}R\llbracket T\rrbracket$
for some profinite space $T$ (with the $G$ action from the left).
Thus, 
\[
S(L)\cong R\llbracket G/N\rrbracket\otimes_{R\llbracket G\rrbracket}\left(R\llbracket G\rrbracket\otimes_{R}R\llbracket T\rrbracket\right)\cong R\llbracket G/N\rrbracket\otimes_{R}R\llbracket T\rrbracket,
\]
which is free as an $R\llbracket G/N\rrbracket$-module, and hence
$T$-acyclic.

The Grothendieck spectral sequence for the composition $F=T\circ S$
now gives 
\[
E^{2}_{p,q}=\operatorname{Tor}^{R\llbracket G/N\rrbracket}_{p}\left(K,\operatorname{Tor}^{R\llbracket G\rrbracket}_{q}(R\llbracket G/N\rrbracket,M)\right)\Longrightarrow\operatorname{Tor}^{R\llbracket G\rrbracket}_{p+q}(K,M).
\]

Now, let $P_{\bullet}\to R$ be a free resolution of $R$ in $\operatorname{PMOD}(R\llbracket N\rrbracket)$.
We get a projective resolution 
\[
R\llbracket G\rrbracket\otimes_{R\llbracket N\rrbracket}P_{\bullet}\longrightarrow R\llbracket G\rrbracket\otimes_{R\llbracket N\rrbracket}R\cong R\llbracket G/N\rrbracket
\]
of $R\llbracket G/N\rrbracket$ in $\operatorname{PMOD}(R\llbracket G\rrbracket)$.
Therefore 
\[
\begin{aligned}\operatorname{Tor}^{R\llbracket G\rrbracket}_{q}(R\llbracket G/N\rrbracket,M) & \cong H_{q}\left(\left(R\llbracket G\rrbracket\otimes_{R\llbracket N\rrbracket}P_{\bullet}\right)\otimes_{R\llbracket G\rrbracket}M\right)\\
 & \cong H_{q}\left(P_{\bullet}\otimes_{R\llbracket N\rrbracket}R\llbracket G\rrbracket\otimes_{R\llbracket G\rrbracket}M\right)\\
 & \cong H_{q}\left(P_{\bullet}\otimes_{R\llbracket N\rrbracket}M\right)=\operatorname{Tor}^{R\llbracket N\rrbracket}_{q}(R,M).
\end{aligned}
\]
 Substituting this identification into the previous spectral sequence
gives 
\[
E^{2}_{p,q}=\operatorname{Tor}^{R\llbracket G/N\rrbracket}_{p}\left(K,\operatorname{Tor}^{R\llbracket N\rrbracket}_{q}(R,M)\right)\Longrightarrow\operatorname{Tor}^{R\llbracket G\rrbracket}_{p+q}(K,M).\qedhere
\]
as required. 
\end{proof}

\textcolor{violet}{}

\begin{lem}[Functoriality of $\mathrm{Tor}$]
\label{lem:TorFunctoriality}Let $H$ be a profinite group, let $f:G\to H$
be a continuous homomorphism, let $K',M$ be $R\llbracket H\rrbracket$-modules
(also regarded as $R\llbracket G\rrbracket$-modules via $f$), and
let $\varphi:K\to K'$ be an $\rg$-linear map. Then there is a natural
map
\[
\mathrm{Tor}^{R\llbracket G\rrbracket}_{k}(K,M)\longrightarrow\mathrm{Tor}^{R\llbracket H\rrbracket}_{k}(K',M).
\]
\end{lem}

\begin{proof}
Let $P_{\bullet}$ be a projective resolution of $K$ as an $R\llbracket G\rrbracket$-module,
and let $Q_{\bullet}$ be a free resolution of $K'$ as an $R\llbracket H\rrbracket$-module.
Observe that $Q_{\bullet}$ is a resolution of $K'$ also as an $R\llbracket G\rrbracket$-module,
but it might not be a projective resolution. However, since $P$ is
a projective resolution, the map $\varphi:K\to K'$ lifts to a map
of resolutions $P_{\bullet}\longrightarrow Q_{\bullet}$ by the comparison
theorem \cites[Theorem 2.2.6]{Weibel}. Tensoring this with $M$,
we get a map $P_{\bullet}\t M\longrightarrow Q_{\bullet}\t M$. We
have a natural map $Q_{\bullet}\t M\longrightarrow(Q_{\bullet}\t M)_{H}$,
so we get a map $P_{\bullet}\t M\longrightarrow(Q_{\bullet}\t M)_{H}$,
and hence a map $(P_{\bullet}\t M)_{G}\longrightarrow(Q_{\bullet}\t M)_{H}$.
Taking the homology groups of both complexes, we get a map
\[
\mathrm{Tor}^{R\llbracket G\rrbracket}_{k}(K,M)\longrightarrow\mathrm{Tor}^{R\llbracket H\rrbracket}_{k}(K',M).\qedhere
\]
\end{proof}

\begin{cor}
Suppose $N\trianglelefteqslant G$ is a closed normal subgroup and
that $M$ is an $R\llbracket G/N\rrbracket$-module. Then there is
a map (natural in $K$ and $M$)
\[
\mathrm{Tor}^{R\llbracket G\rrbracket}_{k}(K,M)\longrightarrow\mathrm{Tor}^{R\llbracket G/N\rrbracket}_{k}(K_{N},M).
\]
\end{cor}

\subsection{Operations on Groups}

\begin{lem}[Extensions]
\label{lem:FPExt}Suppose $N\trianglelefteqslant G$ is a closed
normal subgroup. Let $K_{1}$ be an $\rg$-module, and let $K_{2}$
be an $R\llbracket G/N\rrbracket$-module which is flat as an $R$-module.
Consider $K_{1}\t K_{2}$ as an $\rg$-module via the diagonal action
of $G$. Then 
\[
\cbf kG{K_{1}\otimes_{R}K_{2}}\leqslant\sum_{p+q=k}\cbf pN{K_{1}}\cbf q{G/N}{K_{2}}.
\]
In particular, if $N$ and $G/N$ are $\mathrm{FP}_{n}$, then $G$
is $\mathrm{FP}_{n}$.
\end{lem}

\begin{proof}
Let $M$ be a simple $\rg$-module, so that $M$ is a finite-dimensional
vector space over the finite field $\mathbb{F}\coloneqq\mathrm{End}_{G}(M)$.
Consider the spectral sequence given by Lemma \ref{lem:TorSpec} (associated
with $N\trianglelefteqslant G$),
\[
E^{2}_{pq}=H_{p}(G/N,\mathrm{Tor}^{R\llbracket N\rrbracket}_{q}(K_{1}\t K_{2},M))\implies\torg k(K_{1}\t K_{2},M).
\]
Therefore,
\begin{align*}
\dimf\torg k(K_{1}\t K_{2},M) & =\sum_{p+q=k}\dimf E^{\infty}_{pq}\leqslant\sum_{p+q=k}\dimf E^{2}_{pq}\\
 & =\sum_{p+q=k}\dimf H_{p}(G/N,\mathrm{Tor}^{R\llbracket N\rrbracket}_{q}(K_{1}\t K_{2},M)).
\end{align*}
Since $N$ acts trivially on $K_{2}$ and $K_{2}$ is flat as $R$-module,
we get that $\mathrm{Tor}^{R\llbracket N\rrbracket}_{q}(K_{1}\t K_{2},M)=K_{2}\t\mathrm{Tor}^{R\llbracket N\rrbracket}_{q}(K_{1},M)$.
Therefore, since $K_{2}$ is flat as an $R$-module, Lemma \ref{lem:flat-modules}
implies that
\[
\dimf H_{p}(G/N,K_{2}\t\mathrm{Tor}^{R\llbracket N\rrbracket}_{q}(K_{1},M))\leqslant\cbf p{G/N}{K_{2}}\dimf\mathrm{Tor}^{R\llbracket N\rrbracket}_{q}(K_{1},M)).
\]
We thus get
\begin{align*}
\dimf\torg k(K_{1}\t K_{2},M) & \leqslant\sum_{p+q=k}\cbf p{G/N}{K_{2}}\dimf\mathrm{Tor}^{R\llbracket N\rrbracket}_{q}(K_{1},M))\\
 & \leqslant\sum_{p+q=k}\cbf p{G/N}{K_{2}}\cbf qN{K_{1}}\dimf M,
\end{align*}
as needed.
\end{proof}

\begin{rem}
The assumption $K_{2}$ is flat can be weakened to the assumption
either $K_{1}$ or $K_{2}$ is flat. More generally, see Lemma \ref{lem:tensors}.
\end{rem}

\begin{lem}
\label{lem:FPforFinInd}Let $H\leqslant G$ be an open subgroup. Then
\[
\cbf kH{\mathrm{Res}^{G}_{H}(K)}\leqslant[G:H]\cdot\cbf kGK
\]
and
\[
\cbf kGK\leqslant\sum_{p+q=k}\left(([G:H]-1)!\cdot([G:H]!)^{q}\cdot\cbf pH{\mathrm{Res}^{G}_{H}(K)}\right).
\]
In particular, $H$ is $\mathrm{FP}_{n}$ if and only if $G$ is $\mathrm{FP}_{n}$.
\end{lem}

\begin{proof}
We first prove the first inequality. Let $M$ be a simple $R\llbracket H\rrbracket$-module,
so that it is finite-dimensional over a finite field $\mathbb{F}$.
By Shapiro's lemma ,
\[
\mathrm{Tor}^{R\llbracket H\rrbracket}_{k}(\mathrm{Res}^{G}_{H}K,M)\cong\torg k(K,\mathrm{Ind}^{G}_{H}M),
\]
so
\[
\dim\mathrm{Tor}^{R\llbracket H\rrbracket}_{k}(\mathrm{Res}^{G}_{H}K,M)=\dim\torg k(K,\mathrm{Ind}^{G}_{H}M)\leqslant\cbf kGK\dim\mathrm{Ind}^{G}_{H}M=\cbf kGK\cdot[G:H]\cdot\dim M,
\]
as needed.

For the second inequality, recall that $N\coloneqq\bigcap_{g\in G}gHg^{-1}$
is an open normal subgroup of $G$, and $[G:N]\leqslant[G:H]!$. By
the previous case, 
\[
\cbf pN{\mathrm{Res}^{G}_{N}(K)}\leqslant[H:N]\cdot\cbf pH{\mathrm{Res}^{G}_{H}(K)}\leqslant([G:H]-1)!\cdot\cbf pH{\mathrm{Res}^{G}_{H}(K)}
\]
for every nonnegative integer $p$. Since $G/N$ is a finite group,
we have 
\[
\cbf q{G/N}R\leqslant[G:N]^{q}\leqslant([G:H]!)^{q}
\]
for every nonnegative integer $q$ (this may be computed directly,
using the bar resolution. Thus, by Lemma \ref{lem:FPExt} (taking
$K_{1}=K$, $K_{2}=R$), we get
\[
\cbf kGK\leqslant\sum_{p+q=k}\cbf pN{\mathrm{Res}^{G}_{N}(K)}\cbf q{G/N}R\leqslant\sum_{p+q=k}([G:H]-1)!\cdot([G:H]!)^{q}\cdot\cbf pH{\mathrm{Res}^{G}_{H}(K)},
\]
as needed.
\end{proof}

\begin{lem}
\label{lem:FP1Quotients}Let $N\trianglelefteqslant G$ be a closed
normal subgroup. Then 
\begin{align*}
\cbf 0{G/N}{K_{N}} & \leqslant\cbf 0GK,\\
\cbf 1{G/N}{K_{N}} & \leqslant\cbf 1GK.
\end{align*}
In particular, if $G$ is $\mathrm{FP}_{1}$, then $G/N$ is $\mathrm{FP}_{1}$.
\end{lem}

\begin{proof}
Let $M$ is a simple $R\llbracket G/N\rrbracket$-module, so that
it is finite-dimensional over a finite field $\mathbb{F}$, and we
may consider it as an $R\llbracket G\rrbracket$-module by letting
$G$ act via the quotient map. We have $K\otimes_{R\llbracket N\rrbracket}M=K_{N}\t M$,
and $K_{N}\otimes_{R\llbracket G/N\rrbracket}M=K\otimes_{R\llbracket G\rrbracket}M$,
so
\[
\frac{\dim_{\mathbb{F}}(K_{N}\otimes_{R\llbracket G/N\rrbracket}M)}{\dim_{\mathbb{F}}M}=\frac{\dim_{\mathbb{F}}(K\otimes_{R\llbracket G\rrbracket}M)}{\dim_{\mathbb{F}}M}\leqslant\cbf 0GK.
\]
Since $M$ was arbitrary, it follows $\cbf 0{G/N}{K_{N}}\leqslant\cbf 0GK$.

Consider the spectral sequence associated with $N\trianglelefteqslant G$
from Lemma \ref{lem:another-ss}, but observe that now $M$ is the
$R\llbracket G/N\rrbracket$-module and $K$ is the $R\llbracket G\rrbracket$-module.
We have 
\[
E_{2}=\mathrm{Tor}^{R\llbracket G/N\rrbracket}_{p}(M,\mathrm{Tor}^{R\llbracket N\rrbracket}_{q}(R,K))\implies\mathrm{Tor}^{R\llbracket G\rrbracket}_{p+q}(M,K).
\]
The end of the $5$ term exact sequence is 
\[
\mathrm{Tor}^{R\llbracket G\rrbracket}_{1}(M,K)\longrightarrow\mathrm{Tor}^{R\llbracket G/N\rrbracket}_{1}(M,K_{N})\longrightarrow0,
\]
so 
\[
\dimf\mathrm{Tor}^{R\llbracket G/N\rrbracket}_{1}(M,K_{N})\leqslant\dimf\mathrm{Tor}^{R\llbracket G\rrbracket}_{1}(M,K).
\]
Since $M$ was arbitrary (and $\mathrm{Tor}$ is symmetric), we get
\[
\cbf 1{G/N}{K_{N}}\leqslant\cbf 1GK.\qedhere
\]
\end{proof}

\begin{lem}
\label{lem:FPforQuotients}Let $N\trianglelefteqslant G$ be a closed
normal subgroup, and assume $N$ acts trivially on $K$ (so that it
is an $R\llbracket G/N\rrbracket$-module). Then we have the following
recursive formula:
\[
\cbf k{G/N}K\leqslant\cbf kGK+\sum^{k}_{r=2}\cbf{k-r}{G/N}K\cbf{r-1}NR.
\]
In particular, if $N$ is $\mathrm{FP}_{n-1}$ and $G$ is $\mathrm{FP}_{n}$,
then $G/N$ is $\mathrm{FP}_{n}$.
\end{lem}

\begin{proof}
Let $M$ be a simple $R\llbracket G/N\rrbracket$-module, so that
it is finite-dimensional over a finite field $\mathbb{F}$. We let
$G$ act on $M$ via the quotient map. Consider the spectral sequence
associated with $1\to N\to G\to G/N\to1$ from Lemma \ref{lem:another-ss},
\[
E_{2}=\mathrm{Tor}^{R\llbracket G/N\rrbracket}_{p}(K,\mathrm{Tor}^{R\llbracket N\rrbracket}_{q}(R,M))\implies\mathrm{Tor}^{R\llbracket G\rrbracket}_{p+q}(K,M).
\]
We have 
\[
E^{2}_{k0}=\mathrm{Tor}^{R\llbracket G/N\rrbracket}_{k}(K,M),
\]
We need to bound $\dim E^{2}_{k0}$. For $r\geqslant2$, we have an
exact sequence
\[
E^{r+1}_{k0}\longrightarrow E^{r}_{k0}\longrightarrow E^{r}_{k-r,r-1},
\]
so that 
\[
\dimf E^{r}_{k0}\leqslant\dimf E^{r+1}_{k0}+\dimf E^{r}_{k-r,r-1}.
\]
It follows that 
\begin{align*}
\dimf E^{2}_{k0} & \leqslant\dimf E^{\infty}_{k0}+\sum^{k}_{r=2}\dimf E^{r}_{k-r,r-1}\leqslant\dimf E^{\infty}_{k0}+\sum^{k}_{r=2}\dimf E^{2}_{k-r,r-1}\\
 & \leqslant\dimf\mathrm{Tor}^{R\llbracket G\rrbracket}_{k}(K,M)+\sum^{k}_{r=2}\dimf\mathrm{Tor}^{R\llbracket G/N\rrbracket}_{k-r}(K,\mathrm{Tor}^{R\llbracket N\rrbracket}_{r-1}(R,M)).\\
 & \leqslant\left(\cbf kGK+\sum^{k}_{r=2}\cbf{k-r}{G/N}K\cbf{r-1}NR\right)\dimf M.
\end{align*}
Since $M$ was arbitrary, the claim follows.
\end{proof}

\begin{lem}
Suppose $G=N\rtimes Q$ is a semidirect product of some profinite
groups $N,Q$, and assume $N$ acts trivially on $K$. Then $\cbf kQK\leqslant\cbf kGK$.
In particular, if $G$ is $\mathrm{FP}_{n}$, then so is $Q$.
\end{lem}

\begin{proof}
We have a quotient map $\pi:G\to Q$ and a section $s:Q\to G$, so
that $\pi\circ s=\mathrm{id}_{Q}$. Let $M$ be a simple $R\llbracket Q\rrbracket$-module,
also considered as an $\rg$-module via $\pi$. Observe $M$ is finite-dimensional
over the finite field $\mathbb{F}\coloneqq\mathrm{End}_{G}(M)=\mathrm{End}_{Q}(M)$.
Functoriality of $\mathrm{Tor}$ (Lemma \ref{lem:TorFunctoriality})
gives us maps $\pi_{*}:\mathrm{Tor}^{R\llbracket G\rrbracket}_{k}(K,M)\to\mathrm{Tor}^{R\llbracket Q\rrbracket}_{k}(K,M)$
and $s_{*}:\mathrm{Tor}^{R\llbracket Q\rrbracket}_{k}(K,M)\to\mathrm{Tor}^{R\llbracket G\rrbracket}_{k}(K,M)$
such that $\pi_{*}\circ s_{*}=\mathrm{id}_{\mathrm{Tor}^{R\llbracket Q\rrbracket}_{k}(K,M)}$,
so that $\pi_{*}$ is onto, and hence $\dimf\mathrm{Tor}^{R\llbracket Q\rrbracket}_{k}(K,M)\leqslant\dimf\mathrm{Tor}^{R\llbracket G\rrbracket}_{k}(K,M)$.
\end{proof}

\subsubsection{Inverse Limits}
\begin{lem}
\label{lem:CoinvInvlim}Suppose $G=\li_{i\in I}G/N_{i}$ is an inverse
limit of profinite groups. Then $K=\li_{i\in I}K_{N_{i}}$ (where
$K_{N_{i}}$ are the $N_{i}$-coinvariants of $K$).
\end{lem}

\begin{proof}
Denote $C_{i}=\overline{\left\{ nx-x\middle|n\in N_{i},x\in K\right\} }$,
so that $K_{N_{i}}=K/C_{i}$. All we need to prove is that $\bigcap_{i}C_{i}=\left\{ 0\right\} $
(e.g., by \cite[Proposition 1.2.2]{Wil98}). Observe that, for any
submodule $U\subseteq K$, we have $C_{i}\subseteq U$ if and only
if $N_{i}$ acts trivially on $K/U$. Now, let $U$ be any open submodule
of $K$. Then $K/U$ is finite, which means that there is an open
normal subgroup $N\trianglelefteqslant G$ such that $N$ acts trivially
on $K/U$. Since $N$ is open and closed, there is some $i_{0}$ such
that $N_{i_{0}}\subseteq N$. Since $N$ acts trivially on $K/U$,
so does $N_{i_{0}}$, which means $C_{i_{0}}\subseteq U$; in particular,
$\bigcap_{i}C_{i}\subseteq U$. Since $U$ was arbitrary, we get that
$\bigcap_{i}C_{i}\subseteq\bigcap_{U\subseteq K\text{ open}}U=\left\{ 0\right\} $,
as needed.
\end{proof}

\begin{lem}
\label{lem:TorInvLim}Suppose $\Lambda=\li_{i\in I}\Lambda_{i}$ is
an inverse limit of profinite $R$-algebras, that $N=\li_{i\in I}N_{i}$
is a right $\Lambda$-module which is an inverse limit of right $\Lambda_{i}$-modules
$N_{i}$, and that $M=\li_{i\in I}$ is a left $\Lambda$-module which
is an inverse limit of left $\Lambda_{i}$-modules $M_{i}$. Then
\[
\mathrm{Tor}^{\Lambda}_{k}(N,M)\cong\li_{i\in I}\mathrm{Tor}^{\Lambda_{i}}_{k}(N_{i},M_{i}).
\]
\end{lem}

\begin{proof}
Set $N_{0,i}=N_{i}$ and $M_{0,i}=M_{i}$. For $i\in I$, we construct
a free resolution $P_{\bullet,i}\longrightarrow N_{0,i}$ as follows.
Set 
\[
P_{0,i}=\llbracket N_{0,i}\rrbracket\Lambda_{i},
\]
the free profinite right $\Lambda_{i}$-module on the underlying profinite
space of $N_{0,i}$ (ignoring the algebraic structure of $N_{0,i}$).
We let 
\[
\varphi_{0,i}:P_{0,i}\longrightarrow N_{0,i}
\]
be the morphism induced by the identity map $N_{0,i}\longrightarrow N_{0,i}$,
which is onto. Set $N_{1,i}=\ker\varphi_{0,i}$ and $P_{1,i}=\llbracket N_{1,i}\rrbracket\Lambda_{i}$,
and let 
\[
\varphi_{1,i}:P_{1,i}\longrightarrow P_{0,i}
\]
be the map induced by the inclusion $N_{1,i}=\ker\varphi_{0,i}\hookrightarrow P_{0,i}$.
Continue this inductively to define a free resolution 
\[
\cdots\overset{\varphi_{3,i}}{\longrightarrow}P_{2,i}\overset{\varphi_{2,i}}{\longrightarrow}P_{1,i}\overset{\varphi_{1,i}}{\longrightarrow}P_{0,i}\overset{\varphi_{0,i}}{\longrightarrow}N_{0,i}\longrightarrow0
\]
with 
\[
N_{\ell,i}=\ker\varphi_{\ell-1,i},\quad P_{\ell,i}=\llbracket N_{\ell,i}\rrbracket\Lambda_{i}.
\]
Given $i\geqslant j$ in $I$, the map $N_{0,i}\longrightarrow N_{0,j}$
induces a morphism $P_{0,i}\longrightarrow P_{0,j}$. We have a commutative
diagram
\[
\xymatrix{P_{0,i}\ar[r]\ar[d] & N_{0,i}\ar[d]\\
P_{0,j}\ar[r] & N_{0,j}
}
\]
By the commutativity of the diagram, we get a map $N_{1,i}\longrightarrow N_{1,j}$
and hence a morphism $P_{1,i}\longrightarrow P_{1,j}$. Continuing
this recursively, we get a map $N_{\ell,i}\longrightarrow N_{\ell,j}$
for every $\ell\geqslant0$ which induces a morphism $P_{\ell,i}\longrightarrow P_{\ell,j}$.
Taking inverse limits, we define 
\[
N_{(m)}=\li_{i}N_{m,i},\qquad P_{(m)}=\li_{i}P_{m,i}.
\]
By definition, $N_{(0)}=N$. Moreover, we have \cite[Exercise 5.2.3]{RB10}
\[
P_{(m)}\cong\li_{i}\llbracket N_{m,i}\rrbracket\Lambda_{i}\cong\llbracket N_{(m)}\rrbracket\Lambda.
\]
Thus, $P_{(m)}$ is a free profinite right $\Lambda$-module. Since
inverse limits are exact in the category of profinite modules, 
\[
\cdots\longrightarrow P_{(2)}\longrightarrow P_{(1)}\longrightarrow P_{(0)}\longrightarrow N\longrightarrow0
\]
is a free exact resolution of $N$.  We have an exact sequence
\[
P_{(m)}\t\Lambda\t M\overset{d}{\longrightarrow}P_{(m)}\t M\longrightarrow P_{(m)}\otimes_{\Lambda}M\longrightarrow0,
\]
where $d$ satisfies $\alpha\otimes\lambda\otimes\beta\mapsto\alpha\lambda\otimes\beta-\alpha\otimes\lambda\beta$
for $\alpha\in P_{(m)}$, $\beta\in M$, $\lambda\in\Lambda$. Similarly,
for each $i\in I$, we have an exact sequence
\[
P_{m,i}\t\Lambda_{i}\t M_{i}\overset{d_{i}}{\longrightarrow}P_{m,i}\t M_{i}\longrightarrow P_{m,i}\otimes_{\Lambda_{i}}M_{i}\longrightarrow0,
\]
where $d_{i}$ satisfies $\alpha\otimes\lambda\otimes\beta\mapsto\alpha\lambda\otimes\beta-\alpha\otimes\lambda\beta$
for $\alpha\in P_{m,i}$, $\beta\in M_{i}$, $\lambda\in\Lambda_{i}$.
Since taking inverse limits is exact and $d=\li_{i\in I}d_{i}$ (and
completed tensor products commute with inverse limits), we get that
\[
P_{(m)}\otimes_{\Lambda}M\cong\li_{i\in I}P_{m,i}\otimes_{\Lambda_{i}}M_{i}.
\]
This isomorphism is natural, so it is compatible with the differentials
of $P_{(\bullet)}$, and hence 
\[
P_{(\bullet)}\otimes_{\Lambda}M\cong\li_{i}\bigl(P_{\bullet,i}\otimes_{\Lambda_{i}}M_{i}\bigr)
\]
as complexes of profinite $R$-modules. Again using exactness of inverse
limits in the category of profinite modules, homology commutes with
this inverse limit. Therefore 
\[
\begin{aligned}\operatorname{Tor}^{\Lambda}_{n}(N,M) & \cong H_{n}\bigl(P_{(\bullet)}\otimes_{\Lambda}M\bigr)\\
 & \cong\li_{i}H_{n}\bigl(P_{\bullet,i}\otimes_{\Lambda_{i}}M_{i}\bigr)\\
 & \cong\li_{i}\operatorname{Tor}^{\Lambda_{i}}_{n}(N_{i},M_{i}).
\end{aligned}
\]
This is the desired natural isomorphism. 
\end{proof}

\begin{lem}
\label{lem:NFPnForInvLim}Suppose $G=\underset{\longleftarrow}{\lim}_{i\in I}\,(G/N_{i})$
is an inverse limit of profinite groups. Then
\[
\cbf kGK\leqslant\sup_{i\in I}\cbf k{G/N_{i}}{K_{N_{i}}}.
\]
\end{lem}

\begin{proof}
By Lemma \ref{lem:CoinvInvlim}, we have $K=\li_{i}K_{N_{i}}$.
Let $M$ be a finite $\rg$-module, so that it is finite-dimensional
over $\mathbb{F}\coloneqq\mathrm{End}_{G}(M)$. Since $M$ is finite,
there is some $i_{0}\in I$ such that $N_{i_{0}}$ acts trivially
on $M$, and hence $M$ is a $G/N_{i_{0}}$-module. By Lemma \ref{lem:TorInvLim}
(treating $K$ and $K_{N_{i}}$ as right modules in the usual way),
we get 
\[
\torg k(K,M)\cong\li_{i\geqslant i_{0}}\mathrm{Tor}^{R\llbracket G/N_{i}\rrbracket}_{k}(K_{N_{i}},M).
\]
For ease of notation, we write again $I=\left\{ i\in I\middle|i\geqslant i_{0}\right\} $.
We have
\[
\dimf\mathrm{Tor}^{R\llbracket G/N_{i}\rrbracket}_{k}(K_{i},M)\leqslant\sup_{i\in I}\cbf k{G/N_{i}}{K_{N_{i}}}\cdot\dimf M,
\]
and therefore 
\[
\dimf\torg k(K,M)=\dimf\left(\li_{i}\mathrm{Tor}^{R\llbracket G/N_{i}\rrbracket}_{k}(K_{i},M)\right)\leqslant\sup_{i\in I}\cbf k{G/N_{i}}{K_{N_{i}}}\cdot\dimf M.
\]
Since $M$ was arbitrary, we are done.
\end{proof}

\subsubsection{Amalgams and HNN Extensions}

\begin{lem}
Let $G=G_{1}\coprod_{H}G_{2}$ be a proper amalgamated free profinite
product of profinite groups. Then
\begin{align*}
\cbf kGK & \leqslant\cbf k{G_{1}}K+\cbf k{G_{2}}K+\cbf{k-1}HK,
\end{align*}
In particular, if $G_{1}$ and $G_{2}$ are $\mathrm{FP}_{n}$ and
$H$ is $\mathrm{FP}_{n-1}$, then $G$ is $\mathrm{FP}_{n}$.
\end{lem}

\begin{proof}
Let $M$ be a simple $\rg$-module. Let $P_{\bullet}\to K$ be a resolution
of $K$ by free$\rg$-modules. We may also consider it as a resolution
by free $R\llbracket L\rrbracket$-modules for $L=G_{1},G_{2},H$,
since a free $\rg$-module is also a free $R\llbracket L\rrbracket$-module
for every closed subgroup $L\leqslant G$ \cite[Corollary 5.7.2]{RB10}.
For $m\geqslant0$, set 
\[
B_{m}=P_{m}\t M
\]
with the diagonal $G$-action. Observe that $H_{p}(L,B_{m})=0$ for
every $m\geqslant0$, $p>0$ and $L\leqslant G$. Indeed, $B_{m}$
is isomorphic to $P_{m}\t M$ endowed with the left $G$-action, ignoring
the $G$-action on $M$ \cite[Corollary 5.8.2]{RB10}. Writing $P_{m}=R\llbracket L\times T_{m,L}\rrbracket$
for some profinite space $T_{m,L}$, we get that $P_{m}\t M\cong\mathrm{Ind}^{L}_{1}(R\llbracket T_{m,L}\rrbracket\t M)$.
Therefore, by Shapiro's lemma, $H_{p}(L,B_{m})=0$. 

We now apply the profinite Mayer–Vietoris sequence in homology for
the proper amalgamated free profinite product $G=G_{1}\amalg_{H}G_{2}$,
with coefficients in $B_{m}$ \cite[Proposition 3.9]{Coo16}. Since
all positive homology groups of $H,G_{1},G_{2},G$ with coefficients
in $B_{m}$ vanish, we get a short exact sequence 
\[
0\longrightarrow P_{m}\otimes_{R\llbracket H\rrbracket}M\longrightarrow\bigl(P_{m}\otimes_{R\llbracket G_{1}\rrbracket}M\bigr)\oplus\bigl(P_{m}\otimes_{R\llbracket G_{2}\rrbracket}M\bigr)\longrightarrow P_{m}\te M\longrightarrow0.
\]
This sequence is natural in the coefficient module (see \cite[Proposition 3.9]{Coo16}),
so these short exact sequences are compatible with the differentials
of $P_{\bullet}$. We thus obtain a short exact sequence of chain
complexes 
\[
0\longrightarrow P_{\bullet}\otimes_{R\llbracket H\rrbracket}M\longrightarrow\bigl(P_{\bullet}\otimes_{R\llbracket G_{1}\rrbracket}M\bigr)\oplus\bigl(P_{\bullet}\otimes_{R\llbracket G_{2}\rrbracket}M\bigr)\longrightarrow P_{\bullet}\te M\longrightarrow0.
\]
This short exact sequence of chain complexes gives us the following
long exact sequence of homology groups
\[
\cdots\longrightarrow\operatorname{Tor}^{R\llbracket H\rrbracket}_{k}(K,M)\longrightarrow\operatorname{Tor}^{R\llbracket G_{1}\rrbracket}_{k}(K,M)\oplus\operatorname{Tor}^{R\llbracket G_{2}\rrbracket}_{k}(K,M)\longrightarrow\operatorname{Tor}^{R\llbracket G\rrbracket}_{k}(K,M)\overset{}{\longrightarrow}\operatorname{Tor}^{R\llbracket H\rrbracket}_{k-1}(K,M)\longrightarrow\cdots
\]
(where we set $\mathrm{Tor}_{-1}=0$). Taking $\mathbb{F}=\mathrm{End}_{G}(M)$
(so that $M$ is a finite-dimensional vector space over $\mathbb{F}$),
it follows that
\begin{align*}
\dimf\operatorname{Tor}^{R\llbracket G\rrbracket}_{k}(K,M) & \leqslant\dimf\operatorname{Tor}^{R\llbracket H\rrbracket}_{k-1}(K,M)+\dimf\operatorname{Tor}^{R\llbracket G_{1}\rrbracket}_{k}(K,M)+\dimf\operatorname{Tor}^{R\llbracket G_{2}\rrbracket}_{k}(K,M)\\
 & \leqslant\left(\cbf{k-1}HK+\cbf k{G_{1}}K+\cbf k{G_{2}}K\right)\dimf M.
\end{align*}
Since $M$ was arbitrary, we are done.
\end{proof}

\begin{lem}
Suppose $G=\mathrm{HNN}(H,A,f)$ is a proper profinite $\mathrm{HNN}$-extension
of profinite groups. Then
\begin{align*}
\cbf kGK & \leqslant\cbf kHK+\cbf{k-1}AK,
\end{align*}
In particular, if $H$ is $\mathrm{FP}_{n}$ and $A$ is $\mathrm{FP}_{n-1}$,
then $G$ is $\mathrm{FP}_{n}$.
\end{lem}

\begin{proof}
The proof is almost identical to the proof of the previous lemma,
except we use the Mayer–Vietoris sequence of HNN extensions.

Let $M$ be a simple $\rg$-module. Let $P_{\bullet}\to K$ be a resolution
of $K$ by free$\rg$-modules, which is again also a resolution by
free $R\llbracket L\rrbracket$-modules for every closed $L\leqslant G$.
For $m\geqslant0$, set $B_{m}=P_{m}\t M$ with the diagonal $G$-action,
so that $H_{p}(L,B_{m})=0$ for every $m\geqslant0$, $p>0$ and $L\leqslant G$.
We now apply the profinite Mayer–Vietoris sequence in homology for
the proper profinite $\mathrm{HNN}$-extension $G=\mathrm{HNN}(H,A,f)$,
with coefficients in $B_{m}$, which is natural in the coefficient
module \cite[Proposition 3.13]{Coo16}. We get again a short exact
sequence, since all relevant positive homology groups with coefficients
in $B_{m}$ vanish. 
\[
0\longrightarrow P_{m}\otimes_{R\llbracket A\rrbracket}M\longrightarrow P_{m}\otimes_{R\llbracket H\rrbracket}M\longrightarrow P_{m}\te M\longrightarrow0.
\]
Again we get a short exact sequence of chain complexes 
\[
0\longrightarrow P_{\bullet}\otimes_{R\llbracket A\rrbracket}M\longrightarrow P_{\bullet}\otimes_{R\llbracket H\rrbracket}M\longrightarrow P_{\bullet}\te M\longrightarrow0,
\]
which implies a long exact sequence of the corresponding homology
groups
\[
\cdots\longrightarrow\operatorname{Tor}^{R\llbracket A\rrbracket}_{k}(K,M)\longrightarrow\operatorname{Tor}^{R\llbracket H\rrbracket}_{k}(K,M)\longrightarrow\operatorname{Tor}^{R\llbracket G\rrbracket}_{k}(K,M)\overset{}{\longrightarrow}\operatorname{Tor}^{R\llbracket A\rrbracket}_{k-1}(K,M)\longrightarrow\cdots
\]
(where $\mathrm{Tor}_{-1}=0$). Taking $\mathbb{F}=\mathrm{End}_{G}(M)$,
it follows that
\begin{align*}
\dimf\operatorname{Tor}^{R\llbracket G\rrbracket}_{k}(K,M) & \leqslant\dimf\operatorname{Tor}^{R\llbracket H\rrbracket}_{k}(K,M)+\dimf\operatorname{Tor}^{R\llbracket A\rrbracket}_{k-1}(K,M)\\
 & \leqslant\left(\cbf kHK+\cbf{k-1}AK\right)\dimf M.
\end{align*}
Since $M$ was arbitrary, we are done.
\end{proof}

\subsection{\label{subsec:Operations-on-Modules}Operations on Modules}
\begin{lem}[Short-exact-sequences]
\label{lem:SES-lemma}Suppose $0\to K_{1}\to K_{2}\to K_{3}\to0$
is a short exact sequence of profinite $R\llbracket G\rrbracket$-modules.
Then:
\begin{enumerate}
\item If $K_{1}$ and $K_{3}$ are $\mathrm{FP}_{n}$, then so is $K_{2}$.
Numerically, $\cbf kG{K_{2}}\leqslant\cbf kG{K_{1}}+\cbf kG{K_{3}}$.
\item If $K_{1}$ is $\mathrm{FP}_{n-1}$ and $K_{2}$ is $\mathrm{FP}_{n}$,
then $K_{3}$ is $\mathrm{FP}_{n}$. Numerically, $\cbf kG{K_{3}}\leqslant\cbf{k-1}G{K_{1}}+\cbf kG{K_{2}}$.
\item If $K_{2}$ is $\mathrm{FP}_{n}$ and $K_{3}$ is $\mathrm{FP}_{n+1}$,
then $K_{1}$ is $\mathrm{FP}_{n}$. Numerically, $\cbf kG{K_{1}}\leqslant\cbf kG{K_{2}}+\cbf{k+1}G{K_{3}}$.
\end{enumerate}
\end{lem}

\begin{proof}
Let $M$ be a simple $R\llbracket G\rrbracket$-module. Since $\mathrm{Tor}^{R\llbracket G\rrbracket}_{\bullet}(-,M)$
is the derived functor of $-\te M$, we have a long exact sequence
\[
\cdots\to\mathrm{Tor}^{R\llbracket G\rrbracket}_{i+1}(K_{3},M)\to\mathrm{Tor}^{R\llbracket G\rrbracket}_{i}(K_{1},M)\to\mathrm{Tor}^{R\llbracket G\rrbracket}_{i}(K_{2},M)\to\mathrm{Tor}^{R\llbracket G\rrbracket}_{i}(K_{3},M)\to\mathrm{Tor}^{R\llbracket G\rrbracket}_{i-1}(K_{1},M)\to\cdots
\]
(setting $\mathrm{Tor}_{-1}(-,-)=0$). We therefore get the following
inequalities, where $\mathbb{F}=\mathrm{End}_{G}(M)$:
\begin{align*}
\dim_{\mathbb{F}}\mathrm{Tor}^{R\llbracket G\rrbracket}_{i}(K_{2},M) & \leqslant\dim_{\mathbb{F}}\mathrm{Tor}^{R\llbracket G\rrbracket}_{i}(K_{1},M)+\dim_{\mathbb{F}}\mathrm{Tor}^{R\llbracket G\rrbracket}_{i}(K_{3},M),\\
\dim_{\mathbb{F}}\mathrm{Tor}^{R\llbracket G\rrbracket}_{i}(K_{3},M) & \leqslant\dim_{\mathbb{F}}\mathrm{Tor}^{R\llbracket G\rrbracket}_{i}(K_{2},M)+\dim_{\mathbb{F}}\mathrm{Tor}^{R\llbracket G\rrbracket}_{i-1}(K_{1},M),\\
\dim_{\mathbb{F}}\mathrm{Tor}^{R\llbracket G\rrbracket}_{i}(K_{1},M) & \leqslant\dim_{\mathbb{F}}\mathrm{Tor}^{R\llbracket G\rrbracket}_{i+1}(K_{3},M)+\dim_{\mathbb{F}}\mathrm{Tor}^{R\llbracket G\rrbracket}_{i}(K_{2},M).
\end{align*}
The desired numerical statements follow from these inequalities by
definition, and the statements on property $\mathrm{FP}_{n}$ then
follow from Theorem \ref{thm:numericalFP}.
\end{proof}

\begin{lem}
\label{lem:tensors}Suppose $N\trianglelefteqslant G$ is a closed
normal subgroup, and let $K'$ be an $R\llbracket G/N\rrbracket$-module.
Consider $K\t K'$ as an $\rg$-module via the diagonal action of
$G$. Then 
\[
\cbf kG{K\otimes_{R}K'}\leqslant\sum_{p+q=k}\cbf pNK\cbf q{G/N}{K'}+\sum^{k-2}_{\ell=0}\cbf{\ell}G{\mathrm{Tor}^{R}_{k-1-\ell}(K,K')}.
\]
In particular, if $K$ is $\mathrm{FP}_{n}$ as an $R\llbracket N\rrbracket$-module,
$K'$ is $\mathrm{FP}_{n}$ as an $R\llbracket G/N\rrbracket$-module,
and either $K$ or $K'$ is flat as $R$-module, then $K\t K'$ is
$\mathrm{FP}_{n}$ as an $\rg$-module.
\end{lem}

\begin{proof}
We prove this by induction on $k$. For $k=0$, the sum $\sum^{k-1}_{\ell=1}\cbf{k-j-1}G{\mathrm{Tor}^{R}_{j}(K,K')}$
is zero. Let
\[
0\longrightarrow C\longrightarrow P\longrightarrow K'\longrightarrow0
\]
be a short exact sequence of $R\llbracket G/N\rrbracket$-module,
with $P\longrightarrow K'$ a projective cover of $K'$. Since tensoring
with $K$ is right exact, we get a surjection
\[
K\t P\longrightarrow K\t K',
\]
so 
\[
\cbf 0G{K\t K'}\leqslant\cbf 0G{K\t P}.
\]
Now, $P$ is flat as an $R$-module, so we may apply Lemma \ref{lem:FPExt}
and get that
\[
\cbf 0G{K\t P}\leqslant\cbf 0NK\cbf 0{G/N}P.
\]
By Lemma \ref{lem:tensor-of-projective-covers}, $\cbf 0{G/N}P=\cbf 0{G/N}{K'}$,
so we get 
\[
\cbf 0G{K\t K'}\leqslant\cbf 0NK\cbf 0{G/N}{K'},
\]
as needed.

Now, assume the lemma is true for $k-1$ and any $R\llbracket G/N\rrbracket$-module
$K'$. Again take a short exact sequence 
\[
0\longrightarrow C\longrightarrow P\longrightarrow K'\longrightarrow0
\]
of $R\llbracket G/N\rrbracket$-modules, with $P\longrightarrow K'$
a projective cover of $K'$. Since $P$ is $R$-flat, tensoring with
$K$ over $R$ gives us the exact sequence
\[
0\longrightarrow\mathrm{Tor}^{R}_{1}(K,K')\longrightarrow K\t C\longrightarrow K\t P\longrightarrow K\t K'\longrightarrow0.
\]
Take $J=\ker(K\t P\longrightarrow K\t K')$, so that both
\[
0\longrightarrow J\longrightarrow K\t P\longrightarrow K\t K'\longrightarrow0
\]
and
\[
0\longrightarrow\mathrm{Tor}^{R}_{1}(K,K')\longrightarrow K\t C\longrightarrow J\longrightarrow0
\]
are exact. We apply the short exact sequence lemma (Lemma \ref{lem:SES-lemma})
to both of them, and get
\[
\cbf kG{K\t K'}\leqslant\cbf{k-1}GJ+\cbf kG{K\t P}
\]
and 
\[
\cbf{k-1}GJ\leqslant\cbf{k-2}G{\mathrm{Tor}^{R}_{1}(K,K')}+\cbf{k-1}G{K\t C}
\]
(recall that $\cbf k{\cdot}{\cdot}=0$ for negative $k$). Therefore,
\[
\cbf kG{K\t K'}\leqslant\cbf kG{K\t P}+\cbf{k-1}G{K\t C}+\cbf{k-2}G{\mathrm{Tor}^{R}_{1}(K,K')}.
\]
Since $P$ is $R$-flat and since $\cbf{\ell}{G/N}P=0$ for $\ell>0$,
we get (by Lemma \ref{lem:FPExt} and Lemma \ref{lem:tensor-of-projective-covers})
\[
\cbf kG{K\t P}\leqslant\cbf kNK\cbf 0{G/N}{K'}.
\]
As for $\cbf{k-1}G{K\t C}$, we get by the induction hypothesis that
\[
\cbf{k-1}G{K\t C}\leqslant\sum_{p+q=k-1}\cbf pNK\cbf q{G/N}C+\sum^{k-3}_{\ell=0}\cbf{\ell}G{\mathrm{Tor}^{R}_{k-2-\ell}(K,C)}.
\]
By Lemma \ref{lem:tau-of-projective-cover}, we have 
\[
\cbf q{G/N}C\leqslant\cbf{q+1}{G/N}{K'},
\]
so 
\[
\cbf kG{K\t K'}\leqslant\sum_{p+q=k}\cbf pNK\cbf q{G/N}{K'}+\sum^{k-3}_{\ell=0}\cbf{\ell}G{\mathrm{Tor}^{R}_{k-2-\ell}(K,C)}+\cbf{k-2}G{\mathrm{Tor}^{R}_{1}(K,K')}.
\]
We now examine
\[
\cbf{\ell}G{\mathrm{Tor}^{R}_{k-2-\ell}(K,C)}.
\]
Applying $\mathrm{Tor}^{R}$ to $0\to C\to P\to K'\to0$ gives an
exact sequence 
\[
\mathrm{Tor}^{R}_{m+1}(K,P)\longrightarrow\mathrm{Tor}^{R}_{m+1}(K,K')\longrightarrow\mathrm{Tor}^{R}_{m}(K,C)\longrightarrow\mathrm{Tor}^{R}_{m}(K,P).
\]
Since $P$ is flat over $R$, taking $m=k-2-\ell>0$ (for $\ell=0,\dots,k-3$)
implies 
\[
\mathrm{Tor}^{R}_{k-2-\ell}(K,C)\cong\mathrm{Tor}^{R}_{k-2-\ell+1}(K,K').
\]
Thus,
\[
\sum^{k-3}_{\ell=0}\cbf{\ell}G{\mathrm{Tor}^{R}_{k-2-\ell}(K,C)}+\cbf{k-2}G{\mathrm{Tor}^{R}_{1}(K,K')}=\sum^{k-2}_{\ell=0}\cbf{\ell}G{\mathrm{Tor}^{R}_{k-1-\ell}(K,K')}.
\]
Putting it all together, we get
\[
\cbf kG{K\t K'}\leqslant\sum_{p+q=k}\cbf pNK\cbf q{G/N}{K'}+\sum^{k-2}_{\ell=0}\cbf{\ell}G{\mathrm{Tor}^{R}_{k-1-\ell}(K,K')},
\]
as needed.
\end{proof}

\begin{lem}
Assume $R$ is a finite field and that $A$ is a nonzero finite-dimensional
$R\llbracket G\rrbracket$-module. Then $\cbf kG{K\otimes_{R}A}\leqslant\cbf kGK\cdot\dim_{R}A$
and $\cbf kGK\leq\sum_{p+q=k}\frac{(|A|!)^{q+1}}{\dim_{R}A}\,\cbf pG{K\t A}$.
In particular, $K$ is $\mathrm{FP}_{n}$ if and only if $K\otimes_{R}A$
is.
\end{lem}

\begin{proof}
We begin with the first inequality. Since we are working over a field,
\[
\cbf kG{K\otimes_{R}A}=\sup_{\text{simple }R\llbracket G\rrbracket\text{-modules }M}\frac{\dim_{R}H_{k}(G,K\otimes_{R}A\otimes_{R}M)}{\dim_{R}M}.
\]
For every simple $R\llbracket G\rrbracket$-module $M$, we have 
\[
\dim_{R}H_{k}(G,K\otimes_{R}A\otimes_{R}M)\leqslant\cbf kGK\cdot\dim_{R}(A\otimes_{R}M)=\cbf kGK\cdot\dim_{R}A\cdot\dim_{R}M,
\]
so that $\cbf kG{K\otimes_{R}A}\leqslant\cbf kGK\cdot\dim_{R}A$,
as needed.

We now prove the second inequality. Observe that, since $|A|<\infty$,
there is an open normal subgroup $N\trianglelefteqslant G$ of index
at most $|A|!$ such that $N$ acts trivially on $A$. By Lemma~\ref{lem:FPforFinInd},
we have 
\[
\cbf pN{K\t A}\leqslant|A|!\cdot\cbf pG{K\t A}
\]
for every nonnegative integer $p$. Since $N$ acts trivially on $A$,
we have 
\[
H_{p}(N,K\otimes_{R}A\otimes_{R}M)=A\otimes_{R}H_{p}(N,K\otimes_{R}M)
\]
for every $R\llbracket N\rrbracket$-module $M$. Therefore 
\begin{align*}
\cbf pN{K\t A} & =\sup_{\text{simple }R\llbracket N\rrbracket\text{-modules }M}\frac{\dim_{R}H_{p}(N,K\otimes_{R}A\otimes_{R}M)}{\dim_{R}M}=\sup_{\text{simple }R\llbracket N\rrbracket\text{-modules }M}\frac{\dim_{R}\bigl(A\otimes_{R}H_{p}(N,K\otimes_{R}M)\bigr)}{\dim_{R}M}\\
 & =\sup_{\text{simple }R\llbracket N\rrbracket\text{-modules }M}\frac{\dim_{R}A\cdot\dim_{R}\bigl(H_{p}(N,K\otimes_{R}M)\bigr)}{\dim_{R}M}=\dim_{R}A\cdot\cbf pNK
\end{align*}
In other words, $\cbf pNK=\frac{\cbf pN{K\t A}}{\dim_{R}A}\leqslant\frac{|A|!\cdot\cbf pG{K\t A}}{\dim_{R}A}$.
Thus, by Lemma~\ref{lem:FPExt} (taking $K_{1}=K$ and $K_{2}=R$)
we get 
\begin{align*}
\cbf kGK & \leqslant\sum_{p+q=k}\cbf pNK\cbf q{G/N}R\leqslant\sum_{p+q=k}\frac{|A|!\cdot\cbf pG{K\t A}}{\dim_{R}A}\,(|A|!)^{q}\\
 & =\sum_{p+q=k}\frac{(|A|!)^{q+1}}{\dim_{R}A}\cbf pG{K\t A}
\end{align*}
as needed. 
\end{proof}

\begin{lem}
\label{lem:direct-sum}Assume $K=K_{1}\oplus K_{2}$, a direct sum
of $R\llbracket G\rrbracket$-modules. Then $K$ is $\mathrm{FP}_{n}$
if and only if both $K_{1}$ and $K_{2}$ are $\mathrm{FP}_{n}$.
Numerically, 
\[
\cbf kGK\leqslant\cbf kG{K_{1}}+\cbf kG{K_{2}}
\]
\end{lem}

\begin{proof}
For every $R\llbracket G\rrbracket$-module $M$, the functor $\mathrm{Tor}^{R\llbracket G\rrbracket}_{k}(-,M)$
is additive, so $\mathrm{Tor}^{R\llbracket G\rrbracket}_{k}(K,M)\cong\mathrm{Tor}^{R\llbracket G\rrbracket}_{k}(K_{1},M)\oplus\mathrm{Tor}^{R\llbracket G\rrbracket}_{k}(K_{2},M)$.
\end{proof}

\begin{lem}
\label{lem:induction}Let $H\leqslant G$ be a closed subgroup and
let $L$ be an $R\llbracket H\rrbracket$-module. Then 
\[
\cbf kG{\mathrm{Ind}^{G}_{H}L}\leqslant\cbf kHL.
\]
In particular, if $L$ is $\mathrm{FP}_{n}$ as an $R\llbracket H\rrbracket$-module,
then $\mathrm{Ind}^{G}_{H}L$ is $\mathrm{FP}_{n}$ as an $R\llbracket G\rrbracket$-module.
\end{lem}

\begin{proof}
For every finite simple $R\llbracket G\rrbracket$-module $M$, we
have by Shapiro's lemma 
\[
\mathrm{Tor}^{R\llbracket G\rrbracket}_{k}(\mathrm{Ind}^{G}_{H}L,M)\cong\mathrm{Tor}^{R\llbracket H\rrbracket}_{k}(L,\mathrm{Res}^{G}_{H}M),
\]
so 
\[
\cbf kG{\mathrm{Ind}^{G}_{H}L}\leqslant\cbf kHL.\qedhere
\]

\end{proof}

\begin{lem}
Assume $R$ is a finite field, let $G_{1},G_{2}$ be profinite groups,
and let $K_{i}$ be a nonzero profinite $R\llbracket G_{i}\rrbracket$-module
for $i=1,2$. Consider $K_{1}\otimes_{R}K_{2}$ as an $R\llbracket G_{1}\times G_{2}\rrbracket$-module
via the diagonal action. Then $K_{1}\otimes_{R}K_{2}$ is an $\mathrm{FP}_{n}$
$R\llbracket G_{1}\times G_{2}\rrbracket$-module if and only if $K_{1}$
is an $\mathrm{FP}_{n}$ $R\llbracket G_{1}\rrbracket$-module and
$K_{2}$ is an $\mathrm{FP}_{n}$ $R\llbracket G_{2}\rrbracket$-module.
Numerically, 
\[
\cbf k{G_{1}\times G_{2}}{K_{1}\t K_{2}}\leqslant\sum_{p+q=k}\left(\cbf p{G_{1}}{K_{1}}\cdot\cbf q{G_{2}}{K_{2}}\right)
\]
and 
\[
\cbf k{G_{i}}{K_{i}}\cdot\cbf 0{G_{3-i}}{K_{3-i}}\leqslant\cbf k{G_{1}\times G_{2}}{K_{1}\t K_{2}}
\]
for $i=1,2$. 
\end{lem}

\begin{proof}
If $K_{1}$ and $K_{2}$ are $\mathrm{FP}_{n}$, then $K_{1}\otimes_{R}K_{2}$
is $\mathrm{FP}_{n}$ by Lemma \ref{lem:FPExt} and 
\[
\cbf k{G_{1}\times G_{2}}{K_{1}\t K_{2}}\leqslant\sum_{p+q=k}\left(\cbf p{G_{1}}{K_{1}}\cdot\cbf q{G_{2}}{K_{2}}\right).
\]

On the other hand, assume $K_{1}$ is nonzero and that $K_{1}\otimes_{R}K_{2}$
is $\mathrm{FP}_{n}$, and let us prove that $K_{2}$ is $\mathrm{FP}_{n}$.
Let $M$ be a simple $R\llbracket G_{2}\rrbracket$-module. Let $S$
be any finite $R\llbracket G_{1}\rrbracket$-module such that 
\[
H_{0}(G_{1},K_{1}\otimes_{R}S)\neq0
\]
(for instance, take $S=T^{*}$ where $T$ is a simple quotient of
$K_{1}$). By the Kunneth formula, 
\[
H_{k}(G_{1}\times G_{2},K_{1}\otimes_{R}S\otimes_{R}K_{2}\otimes_{R}M)\cong\bigoplus_{p+q=k}H_{p}(G_{1},K_{1}\otimes_{R}S)\otimes_{R}H_{q}(G_{2},K_{2}\otimes_{R}M).
\]
In particular, 
\begin{align*}
\dim_{R}H_{k}(G_{2},K_{2}\otimes_{R}M)\cdot\dim_{R}\left(H_{0}(G_{1},K_{1}\t S)\right) & \leqslant\dim_{R}\bigl(H_{0}(G_{1},K_{1}\otimes_{R}S)\otimes_{R}H_{k}(G_{2},K_{2}\otimes_{R}M)\bigr)\\
 & \leqslant\dim_{R}H_{k}(G_{1}\times G_{2},K_{1}\otimes_{R}S\otimes_{R}K_{2}\otimes_{R}M)\\
 & \leqslant\dim_{R}S\cdot\cbf k{G_{1}\times G_{2}}{K_{1}\t K_{2}}\cdot\dim_{R}M.
\end{align*}
Dividing by $\dim_{R}S$ and taking the supremum over all such $S$,
we get 
\[
\dim_{R}H_{k}(G_{2},K_{2}\otimes_{R}M)\cdot\cbf 0{G_{1}}{K_{1}}\leqslant\cbf k{G_{1}\times G_{2}}{K_{1}\t K_{2}}\cdot\dim_{R}M.
\]
Since $M$ was arbitrary, we get 
\[
\cbf k{G_{2}}{K_{2}}\cdot\cbf 0{G_{1}}{K_{1}}\leqslant\cbf k{G_{1}\times G_{2}}{K_{1}\t K_{2}}.
\]
The same argument, interchanging the two factors, gives the corresponding
bound for $K_{1}$. 
\end{proof}

\section{\label{subsec:Good}Good Groups and Modules}

Throughout this section, we withhold the convention that all modules
are profinite, and consider also modules of discrete groups. We let
$n$ be some nonnegative integer or $\infty$.

Let $\mathcal{C}$ be a pseudovariety of finite groups (i.e., a class
of finite groups that is closed under subgroups, quotients and direct
products), and let $p$ be a prime. For a discrete group $\Gamma$,
we denote 
\[
\hat{\Gamma}^{(\mathcal{C})}=\li_{N\trianglelefteqslant\Gamma,\Gamma/N\in\mathcal{C}}\Gamma/N,
\]
the \emph{pro-$\mathcal{C}$-completion }of $\Gamma$. For an abstract
$\mathbb{F}_{p}[\Gamma]$-module $K$, we denote 
\[
\hat{K}^{(\mathcal{C})}=\li_{L\leqslant_{\mathcal{C}}K}K/L,
\]
the \emph{pro-$\mathcal{C}$-completion }of $K$, where $L\leqslant_{\mathcal{C}}K$
means that $L$ is a finite-index submodule of $K$ such that $\Gamma/\mathrm{Fix}(K/L)\in\mathcal{C}$.
If $\mathcal{C}$ is the collection of all finite groups, we omit
it from the notation, and simple write $\hat{\Gamma}$, $\hat{K}$.
If $K$ is finitely presented, then $\hat{K}$ is naturally isomorphic
to $\mathbb{F}_{p}\llbracket\hat{\Gamma}\rrbracket\otimes_{\mathbb{F}_{p}[\Gamma]}K$
\cite[Lemma 4.3]{GarJai}.

Observe that $\hat{K}^{(\mathcal{C})}$ is a profinite $\mathbb{F}_{p}\llbracket\hat{\Gamma}^{(\mathcal{C})}\rrbracket$-module
\cite[Lemma 5.1.4(a)]{RB10}, and recall that $H_{k}(\hat{\Gamma}^{(\mathcal{C})},\hat{K}^{(\mathcal{C})})=\li_{L\leqslant_{\mathcal{C}}K}H_{k}(\hat{\Gamma}^{(\mathcal{C})},K/L)$
for every nonnegative integer $k$ \cite[Corollary 6.1.10]{RB10}.
Now, fixing some $L\leqslant_{\mathcal{C}}K$, the collection $\mathcal{U}_{L}\coloneqq\left\{ N\trianglelefteqslant\Gamma\middle|\Gamma/N\in\mathcal{C},N\leqslant\mathrm{Fix}(K/L)\right\} $
is cofinal in $\left\{ N\trianglelefteqslant\Gamma\middle|\Gamma/N\in\mathcal{C}\right\} $,
so $\hat{\Gamma}^{(\mathcal{C})}=\li_{N\in\mathcal{U}_{L}}\Gamma/N$,
and $H_{k}(\hat{\Gamma}^{(\mathcal{C})},K/L)\cong\li_{N\in\mathcal{U}_{L}}H_{k}(\Gamma/N,K/L)$
\cite[Proposition 6.5.7]{RB10}, which means 
\[
H_{k}(\hat{\Gamma}^{(\mathcal{C})},\hat{K}^{(\mathcal{C})})=\li_{L\leqslant_{\mathcal{C}}K,N\in\mathcal{U}_{L}}H_{k}(\Gamma/N,K/L).
\]
For every $L\leqslant_{\mathcal{C}}K$ and $N\in\mathcal{U}_{L}$,
the pair $(\Gamma\to\Gamma/N,K\to K/L)$ is compatible, so we get
a map
\[
H_{k}(\Gamma,K)\longrightarrow H_{k}(\Gamma/N,K/L)
\]
(\cite[III.8]{Bro82}). Therefore, by definition of inverse limits,
we have a map 
\[
H_{k}(\Gamma,K)\longrightarrow H_{k}(\hat{\Gamma}^{(\mathcal{C})},\hat{K}^{(\mathcal{C})}).
\]
If $M$ is a finite $\mathbb{F}_{p}\llbracket\hat{\Gamma}^{(\mathcal{C})}\rrbracket$-module
(also considered as a finite $\mathbb{F}_{p}[\Gamma]$-module via
$\mathbb{F}_{p}[\Gamma]\to\mathbb{F}_{p}\llbracket\hat{\Gamma}\rrbracket$),
then $\left\{ L\otimes_{\mathbb{F}_{p}}M\middle|L\leqslant_{\mathcal{C}}K\right\} $
is cofinal in $\left\{ S\middle|S\leqslant_{\mathcal{C}}\left(K\otimes_{\mathbb{F}_{p}}M\right)\right\} $,
and thence 
\[
\widehat{K\otimes_{\mathbb{F}_{p}}M}^{(\mathcal{C})}=\hat{K}^{(\mathcal{C})}\otimes_{\mathbb{F}_{p}}M.
\]
We thus get, for every finite $\mathbb{F}_{p}\llbracket\hat{\Gamma}^{(\mathcal{C})}\rrbracket$-module
$M$, a map 
\[
H_{k}(\Gamma,K\otimes_{\mathbb{F}_{p}}M)\longrightarrow H_{k}(\hat{\Gamma}^{(\mathcal{C})},\hat{K}^{(\mathcal{C})}\otimes_{\mathbb{F}_{p}}M).
\]

\begin{defn}
Let $\mathcal{C}$ be a pseudovariety of finite groups, $p$ a prime,
$\Gamma$ a discrete group. An $\mathbb{F}_{p}[\Gamma]$-module $K$
is \emph{($\mathcal{C}$, $n$)-good} if the map 
\[
H_{k}(\Gamma,K\otimes_{\mathbb{F}_{p}}M)\longrightarrow H_{k}(\hat{\Gamma}^{(\mathcal{C})},\hat{K}^{(\mathcal{C})}\otimes_{\mathbb{F}_{p}}M)
\]
is an isomorphism for every finite $\mathbb{F}_{p}\llbracket\hat{\Gamma}^{(\mathcal{C})}\rrbracket$-module
$M$ and every nonnegative integer $k\leqslant n$. We say $K$ is
\emph{($\mathcal{C}$, $n$)-pretty-good }if this map is surjective
for every finite $\mathbb{F}_{p}\llbracket\hat{\Gamma}^{(\mathcal{C})}\rrbracket$-module
$M$ and every nonnegative integer $k\leqslant n$. We omit $\mathcal{C}$
if it is the class of all finite groups, and omit $n$ if it is $\infty$.
\end{defn}

Observe that a group is good in the sense of Serre if and only if
the trivial module $\mathbb{F}_{p}$ is good for every prime $p$.
\begin{prop}
\label{prop:Good}Let $\mathcal{C}$ be a pseudovariety of finite
groups and $p$ a prime number. Let $\Gamma$ be a discrete group
and let $K$ be a $(\mathcal{C},n)$-good module of $\Gamma$, and
let 
\[
\cdots\to P_{1}\to P_{0}\to K
\]
be a projective resolution of $\mathbb{F}_{p}[\Gamma]$-modules. Then
$\cbf k{\hat{\Gamma}^{(\mathcal{C})}}{\hat{K}^{(\mathcal{C})}}\leqslant\mathrm{d}_{\mathbb{F}_{p}[\Gamma]}P_{k}$
for every nonnegative integer $k\leqslant n$. In particular, the
pro-$\mathcal{C}$-completion completion of a $(\mathcal{C},n)$-pretty-good
$\mathrm{FP}_{n}$ module is $\mathrm{FP}_{n}$ and the pro-$\mathcal{C}$-completion
completion of a $(\mathcal{C},n)$-pretty-good $\mathrm{FP}_{n}$
group is $\mathrm{FP}_{n}$.
\end{prop}

\begin{proof}
Let $M$ be a finite simple $\mathbb{F}_{p}\llbracket\hat{\Gamma}^{(\mathcal{C})}\rrbracket$-module
and $k\leqslant n$. Since $K$ is $(\mathcal{C},n)$-pretty-good,
the map
\[
H_{k}(\Gamma,K\otimes_{\mathbb{F}_{p}}M)\longrightarrow H_{k}(\hat{\Gamma}^{(\mathcal{C})},\hat{K}^{(\mathcal{C})}\otimes_{\mathbb{F}_{p}}M)
\]
is surjective. Since we are working over a field, 
\[
H_{k}(\Gamma,K\otimes_{\mathbb{F}_{p}}M)\cong\mathrm{Tor}^{\mathbb{F}_{p}[\Gamma]}_{k}(K,M),
\]
which is the $k^{\text{th}}$ homology group of
\[
\cdots\to P_{1}\otimes_{\mathbb{F}_{p}[\Gamma]}M\to P_{0}\otimes_{\mathbb{F}_{p}[\Gamma]}M\to0,
\]
so that 
\[
\dim_{\mathbb{F}_{p}}H_{k}(\Gamma,K\otimes_{\mathbb{F}_{p}}M)\leqslant\dim_{\mathbb{F}_{p}}\left(P_{k}\otimes_{\mathbb{F}_{p}[\Gamma]}M\right)\leqslant\d_{\mathbb{F}_{p}[\Gamma]}P_{k}\cdot\dim_{\mathbb{F}_{p}}M.\qedhere
\]
\end{proof}

For the case of a trivial module, see \cite[Proposition 3.1]{Jai20}.
\begin{cor}
Let $\Gamma$ be a discrete group and let $K$ be an $\mathbb{F}_{p}[\Gamma]$-module.
If $\mathbb{F}_{p}\llbracket\hat{\Gamma}\rrbracket$ is flat as an
$\mathbb{F}_{p}[\Gamma]$-module and $K$ is $\mathrm{FP}_{n}$, then
$K$ is $n$-good and $\hat{K}$ is $\mathrm{FP}_{n}$.
\end{cor}

\begin{proof}
First consider the case $n=0$. Let $M$ be a finite $\mathbb{F}_{p}\llbracket\hat{\Gamma}\rrbracket$-module,
and set $V=K\otimes_{\mathbb{F}_{p}}M$, so $\hat{V}=\widehat{K\otimes_{\mathbb{F}_{p}}M}=\hat{K}\otimes_{\mathbb{F}_{p}}M.$
Since $K$ is finitely generated and $M$ is finite, $V$ is finitely
generated. Therefore, $V_{\Gamma}$ is finite-dimensional over $\mathbb{F}_{p}$.
We have 
\[
\hat{V}_{\hat{\Gamma}}=\li_{L\leqslant V,N\in\mathcal{U}_{L}}\left(V/L\right)_{\Gamma/N}=\li_{L\leqslant K,N\in\mathcal{U}_{L}}\left(V/L\right)_{\Gamma}.
\]
Since $\hat{V}$ is finitely generated as an $\mathbb{F}_{p}\llbracket\hat{\Gamma}\rrbracket$-module,
$\hat{V}_{\hat{\Gamma}}$ is finite-dimensional over $\mathbb{F}_{p}$,
so it is equal to $(V/L)_{\Gamma}$ for every small enough $L\leqslant K$.
Since $V_{\Gamma}$ is finite, we can take $L$ small enough so that
$L\leqslant\ker(V\longrightarrow V_{\Gamma})$, so that $(V/L)_{\Gamma}=V_{\Gamma}$.
We get that $V_{\Gamma}=\hat{V}_{\hat{\Gamma}}$.

We may assume $n\geqslant1$; in particular, $K$ is finitely presented,
so that $\mathbb{F}_{p}\llbracket\hat{\Gamma}\rrbracket\otimes_{\mathbb{F}_{p}[\Gamma]}K\cong\hat{K}$
\cite[Lemma 4.3]{GarJai}.

Let $M$ be a finite $\mathbb{F}_{p}\llbracket\hat{\Gamma}\rrbracket$-module,
and take a projective resolution 
\[
\cdots\longrightarrow P_{1}\longrightarrow P_{0}\longrightarrow K\longrightarrow0
\]
of $K$ $\mathbb{F}_{p}[\Gamma]$-modules, where $P_{k}$ is finitely
generated for $k\leqslant n$. Since $\mathbb{F}_{p}\llbracket\hat{\Gamma}\rrbracket$
is flat over $\mathbb{F}_{p}[\Gamma]$, we get that 
\[
\widehat{\mathbb{F}_{p}\llbracket\hat{\Gamma}\rrbracket\otimes_{\mathbb{F}_{p}[\Gamma]}P_{n+1}}\longrightarrow\mathbb{F}_{p}\llbracket\hat{\Gamma}\rrbracket\otimes_{\mathbb{F}_{p}[\Gamma]}P_{n}\longrightarrow\cdots\longrightarrow\mathbb{F}_{p}\llbracket\hat{\Gamma}\rrbracket\otimes_{\mathbb{F}_{p}[\Gamma]}P_{0}\longrightarrow\mathbb{F}_{p}\llbracket\hat{\Gamma}\rrbracket\otimes_{\mathbb{F}_{p}[\Gamma]}K\longrightarrow0
\]
is a projective resolution of $\mathbb{F}_{p}\llbracket\hat{\Gamma}\rrbracket\otimes_{\mathbb{F}_{p}[\Gamma]}K\cong\hat{K}$
by profinite $\mathbb{F}_{p}\llbracket\hat{\Gamma}\rrbracket$-modules,
where $\mathbb{F}_{p}\llbracket\hat{\Gamma}\rrbracket\otimes_{\mathbb{F}_{p}[\Gamma]}P_{k}$
is a finitely generated for $k\leqslant n$. (Of course, if $n=\infty$,
the term $\widehat{\mathbb{F}_{p}\llbracket\hat{\Gamma}\rrbracket\otimes_{\mathbb{F}_{p}[\Gamma]}P_{n+1}}$
should be ignored). We tensor both resolutions with $M$. Observe
that, if $n<\infty$, then the image of 
\[
\widehat{\mathbb{F}_{p}\llbracket\hat{\Gamma}\rrbracket\otimes_{\mathbb{F}_{p}[\Gamma]}P_{n+1}}\hat{\otimes}_{\mathbb{F}_{p}\llbracket\hat{\Gamma}\rrbracket}M\longrightarrow\mathbb{F}_{p}\llbracket\hat{\Gamma}\rrbracket\otimes_{\mathbb{F}_{p}[\Gamma]}P_{n}\hat{\otimes}_{\mathbb{F}_{p}\llbracket\hat{\Gamma}\rrbracket}M
\]
 is equal to the image of 
\[
\mathbb{F}_{p}\llbracket\hat{\Gamma}\rrbracket\otimes_{\mathbb{F}_{p}[\Gamma]}P_{n+1}\hat{\otimes}_{\mathbb{F}_{p}\llbracket\hat{\Gamma}\rrbracket}M\longrightarrow\mathbb{F}_{p}\llbracket\hat{\Gamma}\rrbracket\otimes_{\mathbb{F}_{p}[\Gamma]}P_{n}\hat{\otimes}_{\mathbb{F}_{p}\llbracket\hat{\Gamma}\rrbracket}M
\]
since $\mathbb{F}_{p}\llbracket\hat{\Gamma}\rrbracket\otimes_{\mathbb{F}_{p}[\Gamma]}P_{n}\hat{\otimes}_{\mathbb{F}_{p}\llbracket\hat{\Gamma}\rrbracket}M$
is a finite module. This means that taking the profinite completion
of the $(n+1)^{\mathrm{th}}$ term does not change the $n^{\mathrm{th}}$
homology group. We therefore get 
\begin{align*}
\mathrm{Tor}^{\mathbb{F}_{p}\llbracket\hat{\Gamma}\rrbracket}_{k}(\hat{K},M) & \cong H_{k}\left(\left(\mathbb{F}_{p}\llbracket\hat{\Gamma}\rrbracket\otimes_{\mathbb{F}_{p}[\Gamma]}P_{\bullet}\right)\otimes_{\mathbb{F}_{p}\llbracket\hat{\Gamma}\rrbracket}M\right)\\
 & \cong H_{k}\left(P_{\bullet}\otimes_{\mathbb{F}_{p}[\Gamma]}\left(\mathbb{F}_{p}\llbracket\hat{\Gamma}\rrbracket\otimes_{\mathbb{F}_{p}\llbracket\hat{\Gamma}\rrbracket}M\right)\right)\\
 & \cong H_{k}\left(P_{\bullet}\otimes_{\mathbb{F}_{p}[\Gamma]}M\right)\cong\mathrm{Tor}^{\mathbb{F}_{p}[\Gamma]}_{k}(K,M).
\end{align*}
By Lemma \ref{lem:flat-modules}, $\mathrm{Tor}^{\mathbb{F}_{p}\llbracket\hat{\Gamma}\rrbracket}_{k}(\hat{K},M)\cong H_{k}(\mathbb{F}_{p}\llbracket\hat{\Gamma}\rrbracket,\hat{K}\otimes_{\mathbb{F}_{p}\llbracket\hat{\Gamma}\rrbracket}M)$
and $\mathrm{Tor}^{\mathbb{F}_{p}[\Gamma]}_{k}(K,M)\cong H_{k}(\mathbb{F}_{p}[\Gamma],K\otimes_{\mathbb{F}_{p}}M)$,
so $K$ is $n$-good.
\end{proof}

\section{Morphisms of Extensions\label{sec:FibreProducts}}
\begin{notation}
In this section we fix a commutative profinite ring $R$, a profinite
group $G$, and a commutative diagram of the form
\[
\xymatrix{1\ar[r] & A\ar@{=}[d]\ar[r] & B_{0}\ar[d]\ar[r] & C_{0}\ar[r]\ar[d] & 1\\
1\ar[r] & A\ar[r] & B_{1}\ar[r] & C_{1}\ar[r] & 1
}
\]
We also let $n$ be some nonnegative integer or $\infty$, and fix
some integer $k\leqslant n$. 
\end{notation}

\begin{rem}
An important example to keep in mind is that of fibre products. If
\begin{align*}
1 & \longrightarrow N_{1}\longrightarrow G_{1}\overset{\varphi_{1}}{\longrightarrow}Q\longrightarrow1,\\
1 & \longrightarrow N_{2}\longrightarrow G_{2}\overset{\varphi_{2}}{\longrightarrow}Q\longrightarrow1,
\end{align*}
are short exact sequences, then one may form the fibre product $G_{1}\times_{Q}G_{2}=\left\{ (g_{1},g_{2})\middle|\varphi_{1}(g_{1})=\varphi_{2}(g_{2})\right\} $,
and get the following two commutative diagrams:
\[
\xymatrix{1\ar[r] & N_{i}\ar@{=}[d]\ar[r] & G_{1}\times_{Q}G_{2}\ar[d]\ar[r] & G_{3-i}\ar[r]\ar[d] & 1\\
1\ar[r] & N_{i}\ar[r] & G_{i}\ar[r] & Q\ar[r] & 1
}
\]
for $i=1,2$. In fact, if $B_{0}\longrightarrow B_{1}$ is surjective
(equivalently, if $C_{0}\longrightarrow C_{1}$ is surjective), then
it is not hard to see $B_{0}\cong C_{0}\times_{C_{1}}B_{1}$, so this
is the only example of a commutative diagram as above with surjective
columns.
\end{rem}

\begin{thm}
\label{thm:Fibre}Assume $B_{0}\longrightarrow B_{1}$ is surjective,
and let $K$ be an $R\llbracket C_{1}\rrbracket$-module. We have
the following recursive formula:
\[
\cbf k{C_{1}}K\leqslant\cbf{k-1}{B_{0}}K+\cbf k{B_{1}}K+\sum_{2\leqslant r\leqslant k-1}\cbf{k-r}{C_{1}}K\cbf{r-1}AR+\sum_{2\leqslant r\leqslant k}\cbf r{C_{0}}K\cbf{k-r}AR
\]
In particular, if $B_{1}$ and $C_{0}$ are $\mathrm{FP}_{n}$, $B_{0}$
is $\mathrm{FP}_{n-1}$ and $A$ is $\mathrm{FP}_{n-2}$ , then $C_{1}$
is $\mathrm{FP}_{n}$.
\end{thm}

\begin{proof}
Let $M$ be a simple $R\llbracket C_{1}\rrbracket$-module, so that
it is finite-dimensional over $\mathbb{F}\coloneqq\mathrm{End}_{C_{1}}(M)$.
We consider it as a $B_{1}$-module via the quotient map, and we consider
the spectral sequence associated with 
\[
1\to A\to B_{1}\to C_{1}\to1
\]
obtained in Lemma \ref{lem:another-ss}. We have 
\[
E^{2}_{pq}=\operatorname{Tor}^{R\llbracket C_{1}\rrbracket}_{p}\left(K,\operatorname{Tor}^{R\llbracket A\rrbracket}_{q}(R,M)\right)\Longrightarrow\operatorname{Tor}^{R\llbracket B_{1}\rrbracket}_{p+q}(K,M).
\]
Since $A$ acts trivially on $M$, we have $E^{2}_{k0}=\mathrm{Tor}^{R\llbracket C_{1}\rrbracket}_{k}(K,M)$,
whose dimension is what we need to bound with respect to $\dimf M$.
For every $r\geqslant2$, we have the following exact sequence:
\[
E^{r+1}_{k0}\longrightarrow E^{r}_{k0}\longrightarrow E^{r}_{k-r,r-1},
\]
which implies 
\[
\dimf E^{r}_{k0}\leqslant\dimf E^{r+1}_{k0}+\dimf E^{r}_{k-r,r-1},
\]
so that 
\begin{align*}
\dimf E^{2}_{k0} & \leqslant\dimf E^{\infty}_{k0}+\sum_{2\leqslant r\leqslant k}\dimf E^{r}_{k-r,r-1}\\
 & \leqslant\dimf\mathrm{Tor}^{R\llbracket B_{1}\rrbracket}_{k}(K,M)+\sum_{2\leqslant r\leqslant k}\dimf\operatorname{Tor}^{R\llbracket C_{1}\rrbracket}_{k-r}\left(K,\operatorname{Tor}^{R\llbracket A\rrbracket}_{r-1}(R,M)\right)\\
 & \leqslant\left(\cbf k{B_{1}}K+\sum_{2\leqslant r\leqslant k-1}\cbf{k-r}{C_{1}}K\cbf{r-1}AR\right)\dimf M\\
 & +\dimf\operatorname{Tor}^{R\llbracket C_{1}\rrbracket}_{0}\left(K,\operatorname{Tor}^{R\llbracket A\rrbracket}_{k-1}(R,M)\right)
\end{align*}
(where we interpret $\mathrm{Tor}_{k-1}$ as $0$ if $k=0$; we assume
henceforth that $k\geqslant1$). We are left with estimating $E^{2}_{0,k-1}=\operatorname{Tor}^{R\llbracket C_{1}\rrbracket}_{0}\left(K,\operatorname{Tor}^{R\llbracket A\rrbracket}_{k-1}(R,M)\right)$.
Now, consider $M$ as a $B_{0}$-module via the quotient map, and
consider the spectral sequence associated with
\[
1\to A\to B_{0}\to C_{0}\to1
\]
(from Lemma \ref{lem:another-ss}),
\[
\tilde{E}^{2}_{pq}=\operatorname{Tor}^{R\llbracket C_{0}\rrbracket}_{p}\left(K,\operatorname{Tor}^{R\llbracket A\rrbracket}_{q}(R,M)\right)\Longrightarrow\operatorname{Tor}^{R\llbracket B_{0}\rrbracket}_{p+q}(K,M).
\]
Consider $N\coloneqq\ker(C_{0}\to C_{1})$. Since the action on $M$
and $K$ is defined via the quotient maps onto $C_{1}$, we get that
$N$ acts trivially on $K$ and $M$. Moreover, since $A$ acts trivially
on $M$, we also get that $N$ acts trivially on $\mathrm{Tor}^{R\llbracket A\rrbracket}_{\bullet}(R,M)$:
if $x\in N$ and $\tilde{x}\in B_{0}$ is any lift, then the image
of $\tilde{x}$ in $B_{1}$ must be contained in $A$, so that conjugating
$A$ by $\tilde{x}$ is the same thing as conjugating by some element
in $A$ (namely, the image of $\tilde{x}$ in $B_{1}$); since $A$
acts trivially on $\mathrm{Tor}^{R\llbracket A\rrbracket}_{\bullet}(R,M)$,
so does $\tilde{x}$ and hence $x$. Thus, we have an isomorphism
\[
E^{2}_{0,k-1}=\left(K\t\operatorname{Tor}^{R\llbracket A\rrbracket}_{k-1}(R,M)\right)_{C_{1}}\cong\left(K\t\operatorname{Tor}^{R\llbracket A\rrbracket}_{k-1}(R,M)\right)_{C_{0}}=\tilde{E}^{2}_{0,k-1}.
\]
We thus need to bound $\dimf\tilde{E}^{2}_{0,k-1}$. For every $r\geqslant2$,
we have the following exact sequence:
\[
\tilde{E}^{r}_{r,k-r}\to\tilde{E}^{r}_{0,k-1}\to\tilde{E}^{r+1}_{0,k-1},
\]
so 
\[
\dimf\tilde{E}^{r}_{0,k-1}\leqslant\dimf\tilde{E}^{r}_{r,k-r}+\dimf\tilde{E}^{r+1}_{0,k-1},
\]
so 
\begin{align*}
\dimf\tilde{E}^{2}_{0,k-1} & \leqslant\dimf\tilde{E}^{\infty}_{0,k-1}+\sum^{k}_{r=2}\dimf\tilde{E}^{2}_{r,k-r}\\
 & \leqslant\dimf\operatorname{Tor}^{R\llbracket B_{0}\rrbracket}_{k-1}(K,M)+\sum^{k}_{r=2}\dimf\operatorname{Tor}^{R\llbracket C_{0}\rrbracket}_{r}\left(K,\operatorname{Tor}^{R\llbracket A\rrbracket}_{k-r}(R,M)\right)\\
 & \leqslant\left(\cbf{k-1}{B_{0}}K+\sum^{k}_{r=2}\cbf r{C_{0}}K\cbf{k-r}AR\right)\dimf M.
\end{align*}
Since $M$ was arbitrary, putting all this together gives us 
\[
\cbf k{C_{1}}K\leqslant\cbf{k-1}{B_{0}}K+\cbf k{B_{1}}K+\sum^{k-1}_{r=2}\cbf{k-r}{C_{1}}K\cbf{r-1}AR+\sum^{k}_{r=2}\cbf r{C_{0}}K\cbf{k-r}AR
\]
which is finite by our assumptions and the induction hypothesis.
\end{proof}

\begin{rem}
Taking $n=0$ and the trivial module $K=R$, Theorem \ref{thm:Fibre}
becomes a profinite version of \cite[3.28(2)]{Cor06}.
\end{rem}

\begin{thm}
Assume $B_{0}\longrightarrow B_{1}$ is surjective, and let $K$
be an $R\llbracket C_{1}\rrbracket$-module. Then we have the following
recursive formula
\[
\cbf k{C_{0}}K\leqslant\cbf k{B_{0}}K+\cbf{k-1}{B_{1}}K+\sum^{k-1}_{r=2}\cbf{k-r}{C_{0}}K\cbf{r-1}AR+\sum^{k}_{r=2}\cbf r{C_{1}}K\cbf{k-r}AR.
\]
In particular, if $A$ is $\mathrm{FP}_{n-2}$, $B_{0}$ is $\mathrm{FP}_{n}$,
$B_{1}$ is $\mathrm{FP}_{n-1}$, and $\mathrm{C_{1}}$ is $\mathrm{FP}_{n}$,
then $C_{0}$ is $\mathrm{FP}_{n}$.
\end{thm}

\begin{proof}
Let $M$ be a simple $R\llbracket C_{0}\rrbracket$-module, and set
$\mathbb{F}=\mathrm{End}_{C_{0}}M$. Consider $M$ and $K$ as $B_{0}$-modules
via $B_{0}\to C_{0}$. The short-exact-sequence $1\to A\to B_{0}\to C_{0}\to1$
supplies us with a spectral sequence
\begin{align*}
E^{2}_{pq} & =\mathrm{Tor}^{R\llbracket C_{0}\rrbracket}_{p}(K,\mathrm{Tor}^{R\llbracket A\rrbracket}_{q}(R,M))\implies\mathrm{Tor}^{R\llbracket B_{0}\rrbracket}_{p+q}(K,M)
\end{align*}
(Lemma \ref{lem:another-ss}). Observe that $E^{2}_{k0}=\mathrm{Tor}^{R\llbracket C_{0}\rrbracket}_{k}(K,M)$
(since $\mathrm{Tor}^{R\llbracket A\rrbracket}_{0}(R,M)=M$ since
$A$ acts trivially on $M$). For $r\geqslant2$, we have an exact
sequence 
\[
E^{r+1}_{k0}\to E^{r}_{k0}\to E^{r}_{k-r.r-1},
\]
so that 
\[
\dimf E^{r}_{k0}\leqslant\dimf E^{r+1}_{k0}+\dimf E^{r}_{k-r,r-1}.
\]
Since $E^{k+1}_{k0}=E^{\infty}_{k0}$ and $\dimf E^{r}_{pq}\leqslant\dimf E^{2}_{pq}$,
we get
\begin{align*}
\dimf E^{2}_{k0} & \leqslant\dimf E^{\infty}_{k0}+\sum^{k}_{r=2}\dimf E^{r}_{k-r,r-1}\\
 & \leqslant\dimf\mathrm{Tor}^{R\llbracket B_{0}\rrbracket}_{k}(K,M)+\sum^{k}_{r=2}\dimf\mathrm{Tor}^{R\llbracket C_{0}\rrbracket}_{k-r}(K,\mathrm{Tor}^{R\llbracket A\rrbracket}_{r-1}(R,M)).
\end{align*}
Observe that this implies $\cbf k{C_{0}}K\leqslant\cbf k{B_{0}}K+\sum^{k}_{r=2}\cbf{k-r}{C_{0}}K\cbf{r-1}AR$,
but that is not quite good enough, since we do not want to assume
$\cbf{k-1}AR$ is finite (otherwise, we could have just invoked Lemma
\ref{lem:FPforQuotients}). Thus, we need to consider $\mathrm{Tor}^{R\llbracket C_{0}\rrbracket}_{0}(K,\mathrm{Tor}^{R\llbracket A\rrbracket}_{k-1}(R,M))$
more carefully.

Consider $N\coloneqq\ker(C_{0}\to C_{1})$. We first consider the
case $N$ acts nontrivially on $M$, so that $M_{N}=0$. Since $A$
acts trivially on $M$, we have 
\[
\mathrm{Tor}^{R\llbracket A\rrbracket}_{k-1}(R,M)\cong\mathrm{Tor}^{R\llbracket A\rrbracket}_{k-1}(R,\mathbb{F})\otimes_{\mathbb{F}}M
\]
(where $R$ acts on $\mathbb{F}$ via $R\longrightarrow\mathbb{F}=\mathrm{End}_{C_{0}}(M)$
and $A$ acts trivially). Observe that $N$ acts trivially on $\mathrm{Tor}^{R\llbracket A\rrbracket}_{k-1}(R,\mathbb{F})$:
if $x\in N$ and $\tilde{x}\in B_{0}$ is any lift, then the image
of $\tilde{x}$ in $B_{1}$ must be contained in $A$, so that conjugating
$A$ by $\tilde{x}$ is the same thing as conjugating by some element
in $A$ (namely, the image of $\tilde{x}$ in $B_{1}$); since $A$
acts trivially on $\mathrm{Tor}^{R\llbracket A\rrbracket}_{k-1}(R,\mathbb{F})$,
so does $\tilde{x}$ and hence $x$. Thus, 
\[
\left(\mathrm{Tor}^{R\llbracket A\rrbracket}_{k-1}(R,M)\right)_{N}\cong\mathrm{Tor}^{R\llbracket A\rrbracket}_{k-1}(R,\mathbb{F})\otimes_{\mathbb{F}}M_{N}=0.
\]
Since $N$ acts trivially on $\mathrm{Tor}^{R\llbracket A\rrbracket}_{k-1}(R,M)$,
it follows $\mathrm{Tor}^{R\llbracket C_{0}\rrbracket}_{0}(K,\mathrm{Tor}^{R\llbracket A\rrbracket}_{k-1}(R,M))=0$. 

We now consider the case $N$ acts trivially on $M$, so that we get
an action of $C_{1}$ on $M$, and hence an action of $B_{1}$ via
$B_{1}\to C_{1}$. Let
\begin{align*}
\tilde{E}^{2}_{pq} & =\mathrm{Tor}^{R\llbracket C_{1}\rrbracket}_{p}(K,\mathrm{Tor}^{R\llbracket A\rrbracket}_{q}(R,M))\implies\mathrm{Tor}^{R\llbracket B_{1}\rrbracket}_{p+q}(K,M)
\end{align*}
be the spectral sequence associated with $1\to A\to B_{1}\to C_{1}\to1$
from Lemma \ref{lem:another-ss}. As above, $N$ acts on $\mathrm{Tor}^{R\llbracket A\rrbracket}_{k-1}(R,M)$
via conjugations by elements from $A$, and hence acts trivially.
We therefore get 
\[
E^{2}_{0,k-1}=\left(K\otimes_{R}\mathrm{Tor}^{R\llbracket A\rrbracket}_{k-1}(R,M)\right)_{C_{0}}=\left(K\otimes_{R}\mathrm{Tor}^{R\llbracket A\rrbracket}_{k-1}(R,M)\right)_{C_{1}}=\tilde{E}^{2}_{0,k-1}.
\]
We now bound $\tilde{E}^{2}_{0,k-1}$. For every $r\geqslant2$, we
have an exact sequence 
\[
\tilde{E}^{r}_{r,k-r}\to\tilde{E}^{r}_{0,k-1}\to\tilde{E}^{r+1}_{0,k-1},
\]
so
\[
\dimf\tilde{E}^{r}_{0,k-1}\leqslant\dimf\tilde{E}^{r}_{r,k-r}+\dimf\tilde{E}^{r+1}_{0,k-1},
\]
(since $\mathbb{F}=\mathrm{End}_{C_{0}}(M)=\mathrm{End}_{C_{1}}(M)$,
since $C_{0}\longrightarrow C_{1}$ is surjective). Thus, 
\begin{align*}
\dimf\tilde{E}^{2}_{0,k-1} & \leqslant\dimf\tilde{E}^{\infty}_{0,k-1}+\sum^{k}_{r=2}\dimf\tilde{E}^{r}_{r,k-r}\\
 & \leqslant\dimf\mathrm{Tor}^{R\llbracket B_{1}\rrbracket}_{k-1}(K,M)+\sum^{k}_{r=2}\dimf\mathrm{Tor}^{R\llbracket C_{1}\rrbracket}_{r}(K,\mathrm{Tor}^{R\llbracket A\rrbracket}_{k-r}(R,M)).\\
 & \leqslant\left(\cbf{k-1}{B_{1}}K+\sum^{k}_{r=2}\cbf r{C_{1}}K\cbf{k-r}AR\right)\dimf M.
\end{align*}
Putting it all together, we get
\[
\cbf k{C_{0}}K\leqslant\cbf k{B_{0}}K+\cbf{k-1}{B_{1}}K+\sum^{k-1}_{r=2}\cbf{k-r}{C_{0}}K\cbf{r-1}AR+\sum^{k}_{r=2}\cbf r{C_{1}}K\cbf{k-r}AR.\qedhere
\]
\end{proof}

The following is a profinite analogue of \cite[Theorem A]{KocLim},
extended to general modules.
\begin{thm}
Let $K$ be an $R\llbracket C_{1}\rrbracket$-module. Then 
\[
\cbf k{B_{1}}K\leqslant\cbf k{B_{0}}K+\sum_{p+q=k,q<k}\left(\cbf p{C_{1}}K\cbf qAR\right).
\]
In particular, if $A$ is $\mathrm{FP}_{n-1}$, $B_{0}$ is $\mathrm{FP}_{n}$
and $C_{1}$ is $\mathrm{FP}_{n}$, then $B_{1}$ is $\mathrm{FP}_{n}$.
\end{thm}

\begin{proof}
Let $M$ be a simple $R\llbracket B_{1}\rrbracket$-module, so that
it is a finite-dimensional vector space over the finite field $\mathbb{F}\coloneqq\mathrm{End}_{B_{1}}(M)$.
First, consider the spectral sequence associated with $1\to A\to B_{1}\to C_{1}\to1$,
\[
E^{2}_{p,q}=\operatorname{Tor}^{R\llbracket C_{1}\rrbracket}_{p}\left(K,\operatorname{Tor}^{R\llbracket A\rrbracket}_{q}(R,M)\right)\Longrightarrow\operatorname{Tor}^{R\llbracket B_{1}\rrbracket}_{p+q}(K,M).
\]
so that 
\[
\dimf\mathrm{Tor}^{R\llbracket B_{1}\rrbracket}_{k}(K,M)=\sum_{p+q=k}\dimf E^{\infty}_{pq}\leqslant\dimf E^{\infty}_{0,k}+\sum_{p+q=k,q<k}\dimf E^{2}_{pq}.
\]
For $q<k$, we have
\[
\dimf E^{2}_{pq}=\dimf\mathrm{Tor}^{R\llbracket C_{1}\rrbracket}_{p}(K,\mathrm{Tor}^{R\llbracket A\rrbracket}_{q}(R,M))\leqslant\cbf p{C_{1}}K\cbf qAR\dimf M.
\]
We now turn to $E^{\infty}_{0,k}$. Consider $M$ as a finite $R\llbracket B_{0}\rrbracket$-module,
and let $\tilde{E}^{2}_{pq}$ be the spectral sequence associated
with $1\to A\to B_{0}\to C_{0}\to1$,
\[
\tilde{E}^{2}_{pq}=\mathrm{Tor}^{R\llbracket C_{0}\rrbracket}_{p}(K,\mathrm{Tor}^{R\llbracket A\rrbracket}_{q}(R,M))\implies\mathrm{Tor}^{R\llbracket B_{0}\rrbracket}_{p+q}(K,M).
\]
By naturality of the Grothendieck spectral sequence, the morphism
of extensions from $1\to A\to B_{0}\to C_{0}\to1$ to $1\to A\to B_{1}\to C_{1}\to1$
induces a map of spectral sequences $\varphi^{2}_{pq}:\tilde{E}^{2}_{pq}\longrightarrow E^{2}_{pq}$.
Since $A$ acts trivially on $\mathrm{Tor}^{R\llbracket A\rrbracket}_{k}(R,M)=H_{k}(A,M)$
and the $B_{0}$-action is induced from $B_{1}$, the map $\varphi^{2}_{0,k}:\tilde{E}^{2}_{0,k}\to E^{2}_{0,k}$
is surjective, as it is simply the natural map from $\left(K\t H_{k}(A,M)\right)_{C_{0}}$
onto $\left(K\t H_{k}(A,M)\right)_{C_{1}}$. Since $\varphi$ commutes
with differentials, it follows by induction that we have surjective
maps $\varphi^{r}_{0,k}:\tilde{E}^{r}_{0,k}\to E^{r}_{0,k}$ for every
$r\geqslant2$. In particular, $\dimf E^{\infty}_{0,k}\leqslant\dimf\tilde{E}^{\infty}_{0,k}$.
We have 
\[
\dimf\tilde{E}^{\infty}_{0,k}\leqslant\dimf\mathrm{Tor}^{R\llbracket B_{0}\rrbracket}_{k}(K,M)\leqslant\cbf k{B_{0}}K\dimf M.
\]
Putting this all together, we get 
\[
\dimf\mathrm{Tor}^{R\llbracket B_{1}\rrbracket}_{k}(K,M)\leqslant\left(\cbf k{B_{0}}K+\sum_{p+q=k,q<k}\left(\cbf p{C_{1}}K\cbf qAR\right)\right)\dimf M.
\]
Since $M$ was arbitrary, we are done.
\end{proof}

Specialising to fibre products, and using the $n$-$(n+1)$-$(n+2)$
theorem proved below, we get:
\begin{cor}
Let 
\[
1\to N_{1}\to G_{1}\to Q\to1,\quad1\to N_{2}\to G_{2}\to Q\to1
\]
be short exact sequences of profinite groups, let $G_{1}\times_{Q}G_{2}$
be the associated fibre product.
\begin{enumerate}
\item If $N_{1}$ is $\mathrm{FP}_{n}$, $G_{1}$ and $G_{2}$ are $\mathrm{FP}_{n+2}$,
then $G_{1}\times_{Q}G_{2}$ is $\mathrm{FP}_{n+1}$ if and only if
$Q$ is $\mathrm{FP}_{n+2}$.
\item If $N_{1}$ is $\mathrm{FP}_{n}$, $G_{1}$ is $\mathrm{FP}_{n+1}$,
$G_{1}\times_{Q}G_{2}$ is $\mathrm{FP}_{n+2}$ and $Q$ is $\mathrm{FP}_{n+2}$,
then $G_{2}$ is $\mathrm{FP}_{n+2}$.
\item If $N_{1}$ is $\mathrm{FP}_{n}$, $G_{2}$ is $\mathrm{FP}_{n+2}$,
$G_{1}\times_{Q}G_{2}$ is $\mathrm{FP}_{n+1}$ and $Q$ is $\mathrm{FP}_{n+1}$,
then $G_{1}$ is $\mathrm{FP}_{n+1}$.
\end{enumerate}
\end{cor}

\section{Nilpotent Groups\label{sec:Nilpotent}}

By common convention, we call a group ‘pronilpotent’ if it
is an inverse limit of \emph{finite }nilpotent groups.

\begin{example}
If $G$ is a finite cyclic group, then 
\[
\cbf nG{\mathbb{F}}=\begin{cases}
1 & \mathrm{char}\mathbb{F}\mid\left|G\right|\\
0 & \mathrm{char}\mathbb{F}\nmid\left|G\right|
\end{cases}
\]
 for every positive integer $n$ and every finite field $\mathbb{F}$.
\end{example}

\begin{proof}
Let $M$ be a simple $\mathbb{F}\llbracket G\rrbracket$-module. We
prove $\dim H^{n}(G,M)=0$ for $n\geqslant1$ and $\mathrm{char}\mathbb{F}\nmid\left|G\right|$
using Tate cohomology. Recall that these groups are defined by 
\[
H^{i}_{T}(G,M)=\begin{cases}
H^{i}(G,M) & i\geqslant1,\\
M^{G}/\mathrm{Nm}_{G}(M) & i=0,\\
\ker\mathrm{Nm}_{G}/I_{G}M & i=-1,\\
H_{-i-1}(G,M) & i\leqslant-2,
\end{cases}
\]
where $I_{G}$ is the \emph{augmentation ideal}, $I_{G}=\ker(\mathbb{Z}[G]\longrightarrow\mathbb{Z})$,
and $\mathrm{Nm}_{G}:M\to M$ is defined by $\mathrm{Nm}_{G}(m)=\sum_{g\in G}gm$.
Since $G$ is cyclic, we have $H^{i}_{T}(G,M)\cong H^{i+2}_{T}(G,M)$
for every $i\in\mathbb{Z}$. Moreover, since $M$ is finite, $\dim H^{i}(G,M)=\dim H^{i+1}(G,M)$
for every $i\in\mathbb{Z}$. Thus, it is enough to compute $\dim H^{0}_{T}(G,M)=0$.
If the action is nontrivial, then $H^{0}_{T}(G,M)$ is a quotient
of $M^{G}=\left\{ 0\right\} $ (since $M$ is simple), so $\dim H^{0}_{T}(G,M)/\dim M=0$.
If the action is trivial, then $M=\mathbb{F}$ (since $M$ is simple),
and $\mathrm{Nm}_{G}(m)=\left|G\right|\cdot m$. If $\mathrm{char}\mathbb{F}$
does not divide $\left|G\right|$, then $\mathrm{Nm}_{G}(M)=M$, so
$H^{0}_{T}(G,M)=\left\{ 0\right\} $ once again, and we deduce $\cbf nG{\mathbb{F}}=0$
for every $n\geqslant1$. If $\mathrm{char}\mathbb{F}$ does divide
$\left|G\right|$, then $\mathrm{Nm}_{G}(M)=\left\{ 0\right\} $,
so $H^{0}_{T}(G,M)=M$, and therefore $\dim H^{0}_{T}(G,M)/\dim M=1$,
and hence $\cbf nG{\mathbb{F}}=1$.
\end{proof}

\begin{lem}
Torsion-free procyclic groups are $\mathrm{FP}_{\infty}$.
\end{lem}

\begin{proof}
If $G$ is a torsion-free procyclic group, then $G$ can be embedded
in $\hat{\mathbb{Z}}$ (since its $p$-Sylow subgroups are all either
$\mathbb{Z}_{p}$ or zero), and hence $\mathrm{cd}_{p}(G)\leqslant1$
for every prime $p$ (since $\mathrm{cd}_{p}(\hat{\mathbb{Z}})=1$
for every prime $p$, and the cohomological dimension of a subgroup
cannot be bigger than that of the ambient group; see \cites[7.3.1]{RB10}).
Since $G$ is finitely generated, $\cb 1G\leqslant\mathrm{d}(G)<\infty$.
Since $H^{n}(G,M)=0$ for every $n\geqslant2$ and $G$-module $M$,
$\cb nG=0$ for every $n\geqslant2$. Thus, by Theorem \ref{thm:numericalFP},
$G$ is $\mathrm{FP}_{\infty}$.
\end{proof}

\begin{lem}
\label{lem:Polyprocyclic-groups-are-FP8}Polyprocyclic groups are
$\mathrm{FP}_{\infty}$.
\end{lem}

\begin{proof}
By Lemma \ref{lem:FPExt}, it is enough to prove procyclic groups
are $\mathrm{FP}_{\infty}$. Every procyclic group is a quotient of
$\hat{\mathbb{Z}}$ by a torsion-free procyclic subgroup (\cites[Theorem 2.7.2]{RB10}),
hence $\mathrm{FP}_{\infty}$ by the last lemma and Lemma \ref{lem:FPforQuotients}.\footnote{This is the argument given in the proof of \cite[Proposition 4.2]{Coo16}.}
\end{proof}

\begin{lem}
\label{lem:commutator-still-dense}Let $G$ be a topological group,
$H_{1},H_{2}\leqslant G$ subgroups. Suppose $\Lambda_{i}\leqslant H_{i}$
is a dense subgroup for $i=1,2$. Then $\left[\Lambda_{1},\Lambda_{2}\right]$
is dense in $\left[H_{1},H_{2}\right]$.
\end{lem}

\begin{proof}
The argument in \cite[Lemma 4.1]{CohVig} works in this generality.
In more detail: suppose $x\in\left[H_{1},H_{2}\right]$ and $V\ni x$
is an open neighbourhood. By definition, there are $x_{1},\dots,x_{n}\in H_{1}$
and $y_{1},\dots,y_{n}\in H_{2}$ such that $x=[x_{1},y_{1}]\cdots[x_{n},y_{n}]$,
so (by the continuity of the group operations) there are open subsets
$U_{i}\ni x_{i}$ and $V_{i}\ni y_{i}$ (for $i=1,\dots,n$) such
that $[U_{1},V_{1}]\cdots[U_{n},V_{n}]\subseteq V$. By assumption,
there are $u_{i}\in U_{i}\cap\Lambda_{1}$ and $v_{i}\in V_{i}\cap\Lambda_{2}$
(for $i=1,\dots,n$), so $[u_{1},v_{1}]\cdots[u_{n},v_{n}]\in[\Lambda_{1},\Lambda_{2}]\cap V$.
\end{proof}

\begin{thm}
\label{thm:FGNilpotentFP}Let $G$ be a finitely generated profinite
virtually nilpotent group. Then $G$ is $\mathrm{FP}_{\infty}$.
\end{thm}

\begin{proof}
We first show that, if $G$ is actually nilpotent (and not just virtually
nilpotent), then it is polyprocyclic, and hence $\mathrm{FP}_{\infty}$.
Let $G$ be a finitely generated profinite nilpotent group, and let
$\Gamma<G$ be a dense subgroup that is finitely generated as an abstract
group. It is well known that $\Gamma_{n}=[\Gamma,\Gamma_{n-1}]$ is
finitely generated as an abstract group for every $n$. By Lemma \ref{lem:commutator-still-dense},
$\Gamma_{n}$ is dense in $G_{n}=\overline{[G,G_{n-1}]}$, so $G_{n}$
is finitely generated as a profinite group. Therefore, $G_{i}/G_{i+1}$
is finitely generated abelian, hence polyprocyclic (\cite[Proposition 8.2.1]{Wil98}).
Thus $G$ itself is polyprocyclic, since we just showed it admits
a subnormal series with polyprocyclic factors. Therefore, $G$ is
$\mathrm{FP}_{\infty}$ by Lemma \ref{lem:Polyprocyclic-groups-are-FP8}.

Now, if $G$ is a finitely generated profinite virtually nilpotent
group, then it admits an open subgroup $H<G$ that is finitely generated
and nilpotent. By the above, $H$ is $\mathrm{FP}_{\infty}$, so,
by Lemma \ref{lem:FPforFinInd}, $G$ is also $\mathrm{FP}_{\infty}$.
\end{proof}

\begin{lem}
\label{lem:FP1forNilpotent}If $G$ is a pronilpotent group, then
it is $\mathrm{FP}_{1}$ if and only if it is finitely generated.
\end{lem}

\begin{proof}
If $G$ is a finitely generated profinite group, then it is $\mathrm{FP}_{1}$
by \cite[Proposition 3.4]{Coo16}.

Now, let $G$ be a pronilpotent group of type $\mathrm{FP}_{1}$,
and denote by $G_{p}$ the (unique) $p$-Sylow subgroup of $G$, so
that $G=\prod G_{p}$ (\cites[Proposition 2.3.8]{RB10}). By Lemma
\ref{lem:FP1Quotients} and \cites[Theorem 7.8.1]{RB10}, we have
\[
\d(G_{p})=\cbf 1{G_{p}}{\mathbb{F}_{p}}\leqslant\cbf 1G{\mathbb{F}_{p}}\leqslant\cbf 1G{\hat{\mathbb{Z}}}.
\]
Now, let $S\subseteq G$ be some finite subset that projects to a
generating set of $G_{p}$ for every $p$. Set $H=\overline{\left\langle S\right\rangle }$,
so that $\pi_{p}(H)=G_{p}$ for every $p$ (where $\pi_{p}:G\to G_{p}$
is the projection). Thus $H$ is a subdirect product of $G=\prod_{p}G_{p}$,
and hence $H=G$ by Goursat's Lemma (since $G_{p_{0}}$ and $\prod_{q\neq p_{0}}G_{q}$
have no nontrivial isomorphic quotients for any $p_{0}$).
\end{proof}

\begin{thm}
Let $G$ be a profinite group, and, for every prime $p$, let $G_{p}$
be a $p$-Sylow subgroup. Then $\cbf nG{\hat{\mathbb{Z}}}\leqslant\sup_{p}\cbf n{G_{p}}{\hat{\mathbb{Z}}}$.
\end{thm}

\begin{proof}
Since $\cbf nG{\hat{\mathbb{Z}}}=\sup_{p}\cbf nG{\mathbb{F}_{p}}$,
we need to show $\cbf nG{\mathbb{F}_{p}}\leqslant\cbf n{G_{p}}{\mathbb{F}_{p}}$
for every $p$. This follows from the fact the restriction of cohomology
$H^{n}(G,M)\longrightarrow H^{n}(G_{p},M)$ is injective for every
$\mathbb{F}_{p}\llbracket G\rrbracket$-module $M$.
\end{proof}

\section{Irreducible Representations of Direct Products\label{sec:DirectProducts}}

\begin{prop}
\label{prop:MaschkeSurrogate}Let $H$ be a group, let $k$ be a field,
let $U$ be a finite-dimensional representation of $H$ over $k$.
Let $I$ be an index set, and for each $i\in I$ let $W_{i}$ be an
irreducible $H$-subrepresentation of $U$ over $k$. Then there exists
a finite subset $J\subseteq I$ such that 
\[
\sum_{i\in I}W_{i}=\bigoplus_{j\in J}W_{j}
\]
as $H$-subrepresentations of $U$ over $k$ (with the direct sum
being internal).
\end{prop}

\begin{proof}
First, observe there is some finite subset $I'\subseteq I$ such that
$\sum_{i\in I'}W_{i}=\sum_{i\in I}W_{i}$, since $\sum_{i\in I}W_{i}$
is finite dimensional. Therefore, we may assume $I=\{1,\dots,n\}$
and prove by induction on $n$ that there is some $m\leqslant n$
and some $1\leqslant i_{1}\leqslant\cdots\leqslant i_{m}\leqslant n$
such that $\sum^{n}_{i=1}W_{i}=\bigoplus^{m}_{j=1}W_{i_{j}}$. In
the base case $n=1$, we take $m=1$ and $i_{1}=1$.

Suppose now that $n>1$, so by the induction hypothesis there exists
a nonnegative integer $\mu\leq n-1$ and integers $1\leq i_{1}<\dots<i_{\mu}\leq n-1$
such that 
\[
\sum^{n-1}_{i=1}W_{i}=\bigoplus^{\mu}_{j=1}W_{i_{j}}.
\]
If $W_{n}$ is contained in the sum above then taking $m=\mu$ we
are done, so we assume the contrary. Our assumption that $W_{n}$
is irreducible then implies that 
\[
W_{n}\cap\left(\bigoplus^{\mu}_{j=1}W_{i_{j}}\right)=\{0\}
\]
so 
\[
\sum^{n}_{i=1}W_{i}=\sum^{n-1}_{i=1}W_{i}+W_{n}=\left(\bigoplus^{\mu}_{j=1}W_{i_{j}}\right)+W_{n}=\left(\bigoplus^{\mu}_{j=1}W_{i_{j}}\right)\oplus W_{n}
\]
and we can take $m=\mu+1$ and $i_{m}=n$.
\end{proof}

\begin{thm}
\label{thm:DirectProduct}Let $G$ and $H$ be groups, let $\mathbb{F}$
be a finite field, and let $U$ be a finite-dimensional irreducible
representation of $G\times H$ over $\mathbb{F}$. Then there exists
a finite field extension $\mathbb{F}_{0}$ of $\mathbb{F}$, a representation
$W$ of $G$ over $\mathbb{F}_{0}$ and a representation $V$ of $H$
over $\mathbb{F}_{0}$ such that $U\cong W\otimes_{\mathbb{F}_{0}}V$
as representations of $G\times H$ over $\mathbb{F}$.
\end{thm}

\begin{proof}
Considering $U$ as a representation of $G$ over $\mathbb{F}$, we
pick an irreducible subrepresentation $W$ of $U$, and denote by
$\iota\colon W\to U$ the inclusion map. The subspace $H.W=\sum_{h\in H}hW$
is invariant under $G\times H$, and is nonzero, so it is all of $U$.
By Proposition \ref{prop:MaschkeSurrogate}, there are $h_{1},\dots,h_{m}$
such that
\[
U=\bigoplus^{m}_{i=1}h_{i}W\cong W^{m},
\]
since $W\cong h_{i}W$ as $G$-representations for every $i=1,\dots,m$
(since $G$ and $H$ commute). Recall that $\mathrm{Hom}_{G}(W,W)$
is a finite field extension of $\mathbb{F}$, and denote it by $\mathbb{F}_{0}$.
Consider next the vector space 
\[
V\coloneqq\Hom_{G}(W,U)\cong\Hom(W,W^{m})\cong\mathbb{F}^{m}_{0},
\]
on which $H$ acts by postcomposition and $\mathbb{F}_{0}$ acts by
precomposition. Since both $W$ and $V$ have a natural structure
of vector spaces over $\mathbb{F}_{0}$, we may form the tensor product
$W\otimes_{\mathbb{F}_{0}}V$. The action of $G$ on $W$ commutes
with $\mathbb{F}_{0}=\mathrm{End}_{G}(W)$, and the action of $H$
on $V=\mathrm{Hom}_{G}(W,U)$ commutes with $\mathbb{F}_{0}$, so
we have an action of $G\times H$ on $W\otimes_{\mathbb{F}_{0}}V$.
The evaluation map $W\otimes_{\mathbb{F}}V\to U$ factors through
to an $\mathbb{F}$-linear map 
\[
e:W\otimes_{\mathbb{F}_{0}}V\to U
\]
(considering $W\otimes_{\mathbb{F}_{0}}V$ as a vector space of $\mathbb{F}$).
This map commutes with the $G\times H$-action. It is nonzero, because
$W\neq0$ and, for every nonzero $w\in W$, we have 
\[
e(\iota\otimes w)=\iota(w)=w\neq0.
\]
Our assumption that $U$ is an irreducible representation of $G\times H$
implies that $e$ is surjective. Lastly, $\dim_{\mathbb{F}_{0}}V=m$
and $\dim_{\mathbb{F}_{0}}W=\frac{\dim_{\mathbb{F}}W}{\dim_{\mathbb{F}}\mathbb{F}_{0}}$,
so $\dim_{\mathbb{F}_{0}}(W\otimes_{\mathbb{F}_{0}}V)=\frac{m\dim_{\mathbb{F}}W}{\dim_{\mathbb{F}}\mathbb{F}_{0}}$,
and thence $\dim_{\mathbb{F}}(W\otimes_{\mathbb{F}_{0}}V)=m\dim_{\mathbb{F}}W$.
On the other hand, $\dim_{\mathbb{F}}U=\dim_{\mathbb{F}}W^{m}=m\dim_{\mathbb{F}}W$.
Thus, $e$ is an isomorphism of $G\times H$-representations over
$\mathbb{F}$.
\end{proof}

\section{The $n$-$(n+1)$-$(n+2)$ Theorem\label{sec:nn1n2}}

Let 
\[
1\longrightarrow N_{1}\longrightarrow G_{1}\xrightarrow{\varphi_{1}}Q\longrightarrow1,\qquad1\longrightarrow N_{2}\longrightarrow G_{2}\xrightarrow{\varphi_{2}}Q\longrightarrow1
\]
be two short exact sequences of profinite groups. Recall that the
fibre product of $G_{1}$ and $G_{2}$ over $Q$ (with respect to
$\varphi_{1}$ and $\varphi_{2}$) is the subgroup of $G_{1}\times G_{2}$
given by 
\[
G_{1}\times_{Q}G_{2}=\{(g_{1},g_{2})\in G_{1}\times G_{2}\colon\varphi_{1}(g_{1})=\varphi_{2}(g_{2})\}.
\]
Let $\mathrm{pr}_{i}\colon G_{1}\times_{Q}G_{2}\to G_{i}$ for $i\in\{1,2\}$
be the projection on the $i^{\text{th}}$ coordinate, and note that
\[
\Ker(\mathrm{pr}_{1})=\{1\}\times N_{2},\quad\Ker(\mathrm{pr}_{2})=N_{1}\times\{1\}.
\]
We also have 
\[
(G_{1}\times_{Q}G_{2})/\Ker(\mathrm{pr}_{1})\cong\mathrm{Im}(\mathrm{pr}_{1})=G_{1},\quad(G_{1}\times_{Q}G_{2})/\Ker(\mathrm{pr}_{2})\cong\mathrm{Im}(\mathrm{pr}_{2})=G_{2}.
\]
We thus get the following commutative diagram of profinite groups,\[\begin{tikzcd}
	1 && 1 & 1 & \\
	& {N_1\times N_2} & {N_2} & {N_2} \\
	1 & {N_1} & {G_1\times_Q G_2} & {G_2} & 1 \\
	1 & {N_1} & {G_1} & Q & 1 \\
	&& 1 & 1 & 1
	\arrow[from=1-3, to=2-3]
	\arrow[from=1-1, to=2-2]
	\arrow[from=1-4, to=2-4]
	\arrow[from=2-2, to=3-3]
	\arrow[equal, from=2-3, to=2-4]
	\arrow[from=2-3, to=3-3]
	\arrow[from=2-4, to=3-4]
	\arrow[from=3-1, to=3-2]
	\arrow[from=3-2, to=3-3]
	\arrow[equal, from=3-2, to=4-2]
	\arrow[from=3-3, to=3-4]
	\arrow[from=3-3, to=4-3]
	\arrow[from=3-3, to=4-4]
	\arrow[from=3-4, to=3-5]
	\arrow[from=3-4, to=4-4]
	\arrow[from=4-1, to=4-2]
	\arrow[from=4-2, to=4-3]
	\arrow[from=4-3, to=4-4]
	\arrow[from=4-3, to=5-3]
	\arrow[from=4-4, to=4-5]
	\arrow[from=4-4, to=5-4]
	\arrow[from=4-4, to=5-5]
\end{tikzcd}\]with exact columns, rows and diagonal.

\begin{thm}[The $n-(n+1)-(n+2)$ Theorem]
\label{thm:n-n+1-n+2}Let $n$ be a nonnegative integer or $\infty$
and fix an integer $k\leqslant n$. Let 
\[
1\longrightarrow N_{1}\longrightarrow G_{1}\xrightarrow{\varphi_{1}}Q\longrightarrow1,\qquad1\longrightarrow N_{2}\longrightarrow G_{2}\xrightarrow{\varphi_{2}}Q\longrightarrow1
\]
be two short exact sequences of profinite groups, and let $K$ be
an $R\llbracket Q\rrbracket$-module. Then 
\begin{align*}
\cbf k{G_{1}\times_{Q}G_{2}}K & \leqslant\cbf k{G_{1}}K+\sum^{k+1}_{r=2}\cbf rQK\cbf{k+1-r}{N_{1}}R+\sum_{p+q=k,p>0}\cbf p{G_{2}}K\cbf q{N_{1}}R
\end{align*}
In particular:
\begin{enumerate}
\item If $N_{1}$ is $\mathrm{FP}_{n-1}$, $K$ is $\mathrm{FP}_{n}$ as
an $R\llbracket G_{1}\rrbracket$-module, $K$ is $\mathrm{FP}_{n}$
as an $R\llbracket G_{2}\rrbracket$-module, and $K$ is $\mathrm{FP}_{n+1}$
as an $R\llbracket Q\rrbracket$-module, then $K$ is $\mathrm{FP}_{n}$
as an $R\llbracket G_{1}\times_{Q}G_{2}\rrbracket$-module.
\item If $N_{1}$ is $\mathrm{FP}_{n-1}$, $G_{1}$ and $G_{2}$ are $\mathrm{FP}_{n}$,
and $Q$ is $\mathrm{FP}_{n+1}$, then $G_{1}\times_{Q}G_{2}$ is
$\mathrm{FP}_{n}$.
\end{enumerate}
\end{thm}

\begin{proof}
By Theorem \ref{thm:numericalFP}, the numerical statement implies
the corollary. Set $G=G_{1}\times_{Q}G_{2}$, and let $M$ be a simple
$\rg$-module, so that it is finite-dimensional over $\mathbb{F}\coloneqq\mathrm{End}_{G}(M)$. 

Consider the spectral sequence from Lemma (\ref{lem:another-ss})
associated with the short exact sequence $1\to N_{1}\to G\to G_{2}\to1$,
\[
E^{2}_{pq}=\mathrm{Tor}^{R\llbracket G_{2}\rrbracket}_{p}(K,\mathrm{Tor}^{R\llbracket N_{1}\rrbracket}_{q}(R,M))\implies\mathrm{Tor}^{R\llbracket G\rrbracket}_{p+q}(K,M).
\]
We thus have
\[
\dimf\mathrm{Tor}^{R\llbracket G\rrbracket}_{k}(K,M)\leqslant\sum_{p+q=k}\dimf E^{\infty}_{pq}.
\]
Moreover,
\[
\dimf E^{\infty}_{pq}\leqslant\dimf E^{2}_{pq}=\dimf\mathrm{Tor}^{R\llbracket G_{2}\rrbracket}_{p}(K,\mathrm{Tor}^{R\llbracket N_{1}\rrbracket}_{q}(R,M)).
\]
First assume $p\geqslant1$, which means $q\leqslant k-1\leqslant n$.
We have 
\[
\dimf\mathrm{Tor}^{R\llbracket N_{1}\rrbracket}_{q}(R,M)\leqslant\cbf q{N_{1}}R\dimf M.
\]
Similarly, 
\[
\dimf\mathrm{Tor}^{R\llbracket G_{2}\rrbracket}_{p}(K,\mathrm{Tor}^{R\llbracket N_{1}\rrbracket}_{q}(R,M))\leqslant\cbf p{G_{2}}K\dimf\mathrm{Tor}^{R\llbracket N_{1}\rrbracket}_{q}(R,M)\leqslant\cbf p{G_{2}}K\cbf q{N_{1}}R\dimf M.
\]
Thus, it remains to bound the dimension of 
\[
E^{2}_{0k}=\mathrm{Tor}^{R\llbracket G_{2}\rrbracket}_{0}(K,\mathrm{Tor}^{R\llbracket N_{1}\rrbracket}_{k}(R,M))=\left(K\t H_{k}(N_{1},M)\right)_{G_{2}}.
\]

By Clifford theory (applied to the normal subgroup $N_{1}\times N_{2}$
of $G_{1}\times_{Q}G_{2}$), $M=\bigoplus^{m}_{i=1}M_{i}$, where
each $M_{i}$ is a simple $(N_{1}\times N_{2})$-modules, and all
the $M_{i}$'s are conjugate to one another under the $Q$-action.
By Theorem \ref{thm:DirectProduct}, there is a finite extension
$\mathbb{F}_{0}$ of $\mathbb{F}$, simple $N_{1}$-modules $M_{1,1},\dots,M_{m,1}$
over $\mathbb{F}_{0}$, and simple $N_{2}$-modules $M_{1,2},\dots,M_{m,2}$
over $\mathbb{F}_{0}$ such that $M_{i}\cong M_{i,1}\otimes_{\mathbb{F}_{0}}M_{i,2}$,
and hence
\[
M=\bigoplus^{m}_{i=1}(M_{i,1}\otimes_{\mathbb{F}_{0}}M_{i,2}).
\]
Let $i\in\left\{ 1,\dots,m\right\} $. Since $N_{1}$ acts trivially
on $M_{i,2}$ and $N_{2}$ acts trivially on $H_{k}(N_{1},M_{i,1})$,
we get
\[
H_{k}(N_{1},M_{i,1}\otimes M_{i,2})_{N_{2}}\cong(H_{k}(N_{1},M_{i,1})\otimes M_{i,2})_{N_{2}}\cong H_{k}(N_{1},M_{i,1})\otimes(M_{i,2})_{N_{2}}.
\]
Consider first the case $N_{2}$ acts nontrivially on $M$. This means
that $N_{2}$ acts nontrivially on $M_{i}$ for every $i\in\left\{ 1,\dots,m\right\} $
(since they are all conjugate to one another, and therefore if $N_{2}$
acts trivially on one of them, it acts trivially on all of them, and
thence acts trivially on $M=\bigoplus^{m}_{i=1}M_{i}$). This implies
that $M_{i,2}$ is a nontrivial $N_{2}$-module for every $i$, so
that $(M_{i,2})_{N_{2}}=0$, so 
\[
H_{k}(N_{1},M_{i,1}\otimes M_{i,2})_{N_{2}}=0
\]
for every $i$, which means 
\[
H_{k}(N_{1},M)_{N_{2}}=0.
\]
Since $N_{2}$ acts trivially on $K$, this implies that
\[
E^{2}_{0k}=\left(\left(K\t H_{k}(N_{1},M)\right)_{N_{2}}\right)_{Q}=\left(K\t H_{k}(N_{1},M)_{N_{2}}\right)_{Q}=0,
\]
so we are done. We may thus assume that $N_{2}$ acts trivially on
$M$, so that the action of $G_{1}\times_{Q}G_{2}$ on $M$ is inflated
from $G_{1}$. Thus, $N_{2}$ acts trivially on $K\t H_{k}(N_{1},M)$,
and hence 
\[
E^{2}_{0k}=\left(\left(K\t H_{k}(N_{1},M)\right)_{N_{2}}\right)_{Q}=\left(K\t H_{k}(N_{1},M)\right)_{Q}.
\]

In order to bound this, we apply the spectral sequence from Lemma
\ref{lem:another-ss} associated with the short exact sequence $1\to N_{1}\to G_{1}\to Q\to1$,
\[
\bar{E}^{2}_{pq}=\mathrm{Tor}^{R\llbracket Q\rrbracket}_{p}(K,\mathrm{Tor}^{R\llbracket N_{1}\rrbracket}_{q}(R,M))\implies\mathrm{Tor}^{R\llbracket G_{1}\rrbracket}_{p+q}(K,M).
\]
We have
\[
E^{2}_{0k}=\left(K\t H_{k}(N_{1},M)\right)_{Q}=\mathrm{Tor}^{R\llbracket Q\rrbracket}_{0}(K,\mathrm{Tor}^{R\llbracket N_{1}\rrbracket}_{k}(R,M))=\bar{E}^{2}_{0k}.
\]
We now bound $\dimf\bar{E}^{2}_{0k}$. For $r\geqslant2$, we have
an exact sequence
\[
\bar{E}^{r}_{r,k+1-r}\longrightarrow\bar{E}^{r}_{0k}\longrightarrow\bar{E}^{r+1}_{0k},
\]
so 
\[
\dimf\bar{E}^{r}_{0k}\leqslant\dimf\bar{E}^{r+1}_{0k}+\dimf\bar{E}^{r}_{r,k+1-r}.
\]
It follows that
\[
\dimf\bar{E}^{2}_{0k}\leqslant\dimf\bar{E}^{\infty}_{0k}+\sum^{k+1}_{r=2}\dimf\bar{E}^{2}_{r,k+1-r}.
\]
We have $\di\bar{E}^{\infty}_{0k}\leqslant\di\mathrm{Tor}^{R\llbracket G_{1}\rrbracket}_{k}(K,M)\leqslant\cbf k{G_{1}}K\dimf M$.
Moreover,
\[
\dimf E^{2}_{r,k+1-r}=\dimf\mathrm{Tor}^{R\llbracket Q\rrbracket}_{r}(K,\mathrm{Tor}^{R\llbracket N_{1}\rrbracket}_{k+1-r}(R,M))\leqslant\cbf rQK\cbf{k+1-r}{N_{1}}R\dimf M.
\]
We get 
\[
\dimf\bar{E}^{2}_{0k}\leqslant\left(\cbf k{G_{1}}K+\sum^{k+1}_{r=2}\cbf rQK\cbf{k+1-r}{N_{1}}R\right)\dimf M.
\]
Therefore, 
\[
\dimf\mathrm{Tor}^{R\llbracket G\rrbracket}_{k}(K,M)\leqslant\left(\cbf k{G_{1}}K+\sum^{k+1}_{r=2}\cbf rQK\cbf{k+1-r}{N_{1}}R+\sum_{p+q=k,p>0}\cbf p{G_{2}}K\cbf q{N_{1}}R\right)\dimf M.
\]
Since $M$ was arbitrary, we deduce 
\[
\cbf kGK\leqslant\cbf k{G_{1}}K+\sum^{k+1}_{r=2}\cbf rQK\cbf{k+1-r}{N_{1}}R+\sum_{p+q=k,p>0}\cbf p{G_{2}}K\cbf q{N_{1}}R.\qedhere
\]
\end{proof}

\begin{cor}
Let $G=\prod_{i\in I}G_{i}$ be a product of profinite groups over
some index set $I$, and let $n$ be a nonnegative integer or $\infty$.
Then
\[
\cbf kG{\mathbb{F}}\leqslant\sup_{S\subseteq I\text{ finite}}\cbf k{\prod_{i\in S}G_{i}}{\mathbb{F}}
\]
for every integer $k$ and every finite field $\mathbb{F}$.
\end{cor}

\begin{proof}
The case $I=\left\{ 1,2\right\} $ follows from the $n-(n+1)-(n+2)$
Theorem (taking $Q=\left\{ 1\right\} $ and $N_{i}=G_{i}$). The case
$I$ is finite then follows by induction, and the general case then
follows from Lemma \ref{lem:NFPnForInvLim}.
\end{proof}

\section{Virtual Surjection\label{sec:VirtualSurjection}}

We follow Kuckuck's proof that the $n$-$(n+1)$-$(n+2)$ Conjecture
implies the Virtual Surjection Conjecture.

We fix some notation. Consider a product $\prod^{n}_{i=1}G_{i}$ and
a subgroup $P\leqslant\prod^{n}_{i=1}G_{i}$. Given $J\subseteq\left\{ 1,\dots,n\right\} $,
we denote by $\mathrm{pr}_{J}:\prod^{n}_{i=1}G_{i}\to\prod_{i\in J}G_{i}$
the projection. We often identify $\prod_{i\in J}G_{i}$ with the
subgroup
\[
\left\{ (g_{1},\dots,g_{n})\in\prod^{n}_{i=1}G_{i}\middle|g_{i}=1\ \forall i\notin J\right\} 
\]
of $\prod^{n}_{i=1}G_{i}$.
\begin{lem}[{\cites[Lemma 3.3]{Kuc14}}]
\label{lem:ProjIsVirNil}Let $n\geqslant k\geqslant2$ be integers,
let $\Gamma_{1},\dots,\Gamma_{n}$ be groups and let $P\leqslant\Gamma_{1}\times\cdots\times\Gamma_{n}$
be a subgroup that projects onto $\Gamma_{i}$ for every $i$, and
projects onto a finite-index subgroup of $\Gamma_{i_{1}}\times\cdots\times\Gamma_{i_{k}}$
for any $1\leqslant i_{1}<\cdots<i_{k}\leqslant n$. Then for every
index $i$ the group $\Gamma_{i}/(P\cap\Gamma_{i})$ is virtually
nilpotent. 
\end{lem}

\begin{lem}[{Goursat's Lemma \cite[Lemma 2.3]{Kuc14}}]
\label{lem:SubDirectAreFibre}Let $P\leqslant G_{1}\times G_{2}$
be a subgroup that projects onto $G_{1}$ and $G_{2}$. Then $P\cap G_{i}$
is normal in $G_{i}$ for $i=1,2$, $G_{1}/(P\cap G_{1})$ is isomorphic
to $G_{2}/(P\cap G_{2})$, and $P$ is the fibre product associated
with the short exact sequences
\begin{align*}
1 & \to P\cap G_{1}\to G_{1}\to G_{1}/(P\cap G_{1})\to1,\\
1 & \to P\cap G_{2}\to G_{2}\to G_{2}/(P\cap G_{2})\to1.
\end{align*}
\end{lem}

\begin{lem}
\label{lem:VirSur}Let $n\geqslant k\geqslant1$ be integers, let
$G_{1},\dots,G_{n}$ be profinite groups of type $\mathrm{FP}_{k}$,
and let $P\leqslant G_{1}\times\cdots\times G_{n}$ be a closed subgroup
such that:
\begin{enumerate}
\item For every $1\leqslant i_{1}<\cdots<i_{k}\leqslant n$, the projection
of $P$ to $G_{i_{1}}\times\cdots\times G_{i_{k}}$ is open in $G_{i_{1}}\times\cdots\times G_{i_{k}}$
(‘$P$ virtually surjects onto $k$-tuples’), and
\item For every $1\leqslant i\leqslant n$, $\mathrm{pr}_{i}(P)/(P\cap G_{i})$
is virtually nilpotent.
\end{enumerate}
Then $P$ is $\mathrm{FP}_{k}$.
\end{lem}

\begin{proof}
A direct computation shows that $P\cap G_{i}$ is always normal in
$\mathrm{pr}_{i}(P)$, so $\mathrm{pr}_{i}(P)/(P\cap G_{i})$ is always
a group. Note that $\mathrm{pr}_{i}(P)$ is open in $G_{i}$, hence
it is $\mathrm{FP}_{k}$ (by Lemma \ref{lem:FPforFinInd}); moreover,
$P\cap G_{i}=P\cap\mathrm{pr}_{i}(P)$, so that $\mathrm{pr}_{i}(P)/(P\cap\mathrm{pr}_{i}(P))=\mathrm{pr}_{i}(P)/(P\cap G_{i})$
is virtually nilpotent. Thus, we may assume without loss of generality
that $\mathrm{pr}_{i}(P)=G_{i}$.

We prove the claim by induction on $k$.

\paragraph{The case $k=1$.}

This we prove by induction on $n$. The case $n=1$ follows from Lemma
\ref{lem:FPforFinInd}. Assume the claim is true for $n-1$ for $n\geqslant2$,
and let us prove it for $n$. Set $T=\mathrm{pr}_{\left\{ 1,\dots,n-1\right\} }(P)\leqslant\prod^{n-1}_{i=1}G_{i}$.
Then $P$ is a subgroup of $T\times G_{n}$ that projects onto each
factor, and hence (by Lemma \ref{lem:SubDirectAreFibre}) it is the
fibre product associated with the sequences
\begin{align*}
1 & \to P\cap T\to T\to T/(P\cap T)\to1,\\
1 & \to P\cap G_{n}\to G_{n}\to G_{n}/(P\cap G_{n})\to1.
\end{align*}
Observe that $T$ is a subgroup of $\prod^{n-1}_{i=1}G_{i}$ that
projects onto every $G_{i}$. Moreover, $T\cap G_{i}$ contains $P\cap G_{i}$
for every $i=1,\dots,n-1$, so $G_{i}/(T\cap G_{i})$ is virtually
nilpotent for every $i=1,\dots,n-1$. Thus, by the induction hypothesis,
$T$ is $\mathrm{FP}_{1}$. Set 
\[
Q=G_{n}/(P\cap G_{n})\cong T/(P\cap T).
\]
Then $Q$ is virtually nilpotent by assumption, $\mathrm{FP}_{1}$
by Lemma \ref{lem:FP1Quotients} (as it is a quotient of $G_{n}$),
hence finitely generated by Lemma \ref{lem:FP1forNilpotent}, hence
it is $\mathrm{FP}_{\infty}$ by Theorem \ref{thm:FGNilpotentFP}.
Now, $P$ is the fibre product of $T\times_{Q}G_{n}$, so by the 0-1-2
Lemma (or our Theorem \ref{thm:n-n+1-n+2}), we get that $P$ is $\mathrm{FP}_{1}$.

\subparagraph{The inductive step.}

Now, assume the claim is true for $k-1$ for $k\geqslant2$, and let
us prove it for $k$. This too we prove by induction on $n$. The
base case $n=k$ again follows from Lemma \ref{lem:FPforFinInd}.
Now, assume the claim is true for $n-1$, where $n\geqslant k+1$,
and let us prove it for $n$.

Once again, set $T=\mathrm{pr}_{\left\{ 1,\dots,n-1\right\} }(P)\leqslant\prod^{n-1}_{i=1}G_{i}$.
First, observe that $\mathrm{pr}_{J}(T)=\mathrm{pr}_{J}(P)$ for every
subset $J\subseteq\left\{ 1,\dots,n-1\right\} $ of size $k$, so
that, by the induction hypothesis on $n$, we get that $T$ is $\mathrm{FP}_{k}$.
Now, as before, $P$ is a subgroup of $T\times G_{n}$ that projects
onto each factor, and hence (by Lemma \ref{lem:SubDirectAreFibre}),
it is the fibre product associated with the sequences
\begin{align*}
1 & \to P\cap T\to T\to T/(P\cap T)\to1,\\
1 & \to P\cap G_{n}\to G_{n}\to G_{n}/(P\cap G_{n})\to1.
\end{align*}
We want to use the profinite $n$-$(n+1)$-$(n+2)$ Theorem. For that,
we need to show $N\coloneqq P\cap\prod^{n-1}_{i=1}G_{i}=P\cap T$
is $\mathrm{FP}_{k-1}$. It is easy to see (see Lemma \ref{lem:KernelOfSubdirect}
below) that $\mathrm{pr}_{J}(N)$ is open in $\prod_{i\in J}G_{i}$
for every $J\subseteq\left\{ 1,\dots,n-1\right\} $ of size $k-1$.
Moreover, 
\[
G_{i}/(N\cap G_{i})=G_{i}/(P\cap G_{i})
\]
is virtually nilpotent for every $i\in\left\{ 1,\dots,n-1\right\} $.
Therefore, by the induction hypothesis on $k$, $N$ is $\mathrm{FP}_{k-1}$.
At last, set 
\[
Q=G_{n}/(P\cap G_{n})\cong T/(P\cap T).
\]
By the Profinite $n$-$(n+1)$-$(n+2)$ Theorem (Theorem \ref{thm:n-n+1-n+2}),
we get that $P=T\times_{Q}G_{n}$ is $\mathrm{FP}_{k}$.
\end{proof}

\begin{thm}[Virtual Surjection Theorem for Profinite Groups]
\label{thm:VirtualSurjection}Let $n\geqslant k\geqslant2$ be integers,
let $G_{1},\dots,G_{n}$ be profinite groups of type $\mathrm{FP}_{k}$,
and let $P\leqslant G_{1}\times\cdots\times G_{n}$ be a closed subgroup
such that:
\begin{itemize}
\item For every $1\leqslant i_{1}<\cdots<i_{k}\leqslant n$, the projection
of $P$ to $G_{i_{1}}\times\cdots\times G_{i_{k}}$ is open in $G_{i_{1}}\times\cdots\times G_{i_{k}}$
(‘$P$ virtually surjects onto $k$-tuples’).
\end{itemize}
Then $P$ is $\mathrm{FP}_{k}$.
\end{thm}

\begin{proof}
This follows from Lemma \ref{lem:VirSur}, since, by Lemma \ref{lem:ProjIsVirNil},
the extra assumption always holds when $k$ is at least $2$.
\end{proof}

The following lemma is a profinite analogue of \cites(Lemma 3.11){Kuc14}.
\begin{lem}
\label{lem:KernelOfSubdirect}Let $k\geqslant2$ and $n\geqslant k$
be integers. Let $G_{1},\dots,G_{n}$ be profinite groups, $P\leqslant G_{1}\times\cdots\times G_{n}$
a subgroup such that $\mathrm{pr}_{J}(P)$ is open in $\prod_{i\in J}G_{i}$
for every $J\subseteq\left\{ 1,\dots,n\right\} $ of size $k$. Set
$N=P\cap\prod^{n-1}_{i=1}G_{i}$. Then $\mathrm{pr}_{J}(N)$ is open
in $\prod_{i\in J}G_{i}$ for every $J\subseteq\left\{ 1,\dots,n-1\right\} $
of size $k-1$.
\end{lem}

\begin{proof}
Set $J'=J\cup\left\{ n\right\} $, which is a subset of $\left\{ 1,\dots,n\right\} $
of size $k$. First, we show $\mathrm{pr}_{J}(N)$ contains $\mathrm{pr}_{J'}(P)\cap\prod_{i\in J}G_{i}$.\footnote{In fact, equality holds (as is easy to see)}
Let $g\in\mathrm{pr}_{J'}(P)\cap\prod_{i\in J}G_{i}$. Take $p\in P$
such that $g=\mathrm{pr}_{J'}(p)$. Since $\mathrm{pr}_{J'}(p)=g$
and $g\in\prod_{i\in J}G_{i}$, we get that the $n^{\text{th}}$ coordinate
of $p$ is $1$. Thus, $p\in\prod^{n-1}_{i=1}G_{i}$. By assumption,
$p\in P$, so $p\in P\cap\prod^{n-1}_{i=1}G_{i}=N$. This also means
$\mathrm{pr}_{J'}(p)=\mathrm{pr}_{J}(p)$, modulo the identification
of $\prod_{i\in J}G_{i}$ as a subgroup of $\prod_{i\in J'}G_{i}$.
Thus, $g=\mathrm{pr}_{J'}(p)=\mathrm{pr}_{J}(p)\in\mathrm{pr}_{J}(N)$,
as needed.

Now, $\mathrm{pr}_{J'}(P)$ is open in $\prod_{i\in J'}G_{i}$, so
$\mathrm{pr}_{J'}(P)\cap G_{J}$ is open in $G_{J}$. Since $\mathrm{pr}_{J}(N)$
contains it, it too is open in $\prod_{i\in J}G_{i}$, as needed.
\end{proof}

\section{Acknowledgements }

We thank Alexander Lubotzky and Guy Kapon for insightful discussions.
Both authors are co-funded by the European Union (ERC, Function Fields,
101161909). Views and opinions expressed are however those of the
authors only and do not necessarily reflect those of the European
Union or the European Research Council. Neither the European Union
nor the granting authority can be held responsible for them. The second
author is The Dr.~A.~Edward Friedmann Career Development Chair in
Mathematics.

\printbibliography

\vspace{0.5cm}

\noindent{\textsc{Department of Mathematics, Weizmann Institute of Science, 234 Herzl Street, Rehovot 7610001, Israel}}

\vspace{0.5cm}

\noindent{\textit{Email address:} \texttt{tal.cohen@weizmann.ac.il}}

\noindent{\textit{Email address:} \texttt{mark.shusterman@weizmann.ac.il}}

\end{document}